\documentclass[11pt,psamsfonts,reqno]{amsart}
\usepackage{amssymb}
\usepackage{amsthm}
\usepackage{graphicx}
\usepackage[dvips]{psfrag}
\usepackage[active]{srcltx}
\usepackage{amsmath}
\usepackage{amsthm}
\usepackage{amssymb}
\usepackage{amscd}
\usepackage{amsfonts}
\usepackage{amsbsy}
\usepackage{epsfig,afterpage}
\usepackage{graphicx}

\newtheorem {theorem} {Theorem}%[section]

\newtheorem {proposition} [theorem]{Proposition}

\newtheorem {lemma}  [theorem]{Lemma}
\newtheorem {remark} [theorem]{Remark}

\newcommand{\R}{\ensuremath{\mathbb{R}}}
\newcommand{\aS}{\ensuremath{\mathbb{S}}}

\advance\textwidth by +1.0in \advance\textheight by +0.5in
\advance\oddsidemargin by -0.5in \advance\evensidemargin by -0.5in
\advance\topmargin by -0.5in

\title[Limit cycles bifurcating from a degenerate center]
{Limit cycles bifurcating from a degenerate center}

\author[ J. Llibre and C. Pantazi]{ Jaume Llibre$^1$ and Chara Pantazi$^2$}

\address{$^1$ Departament de Matem\`{a}tiques, Universitat
Aut\`{o}noma de Barcelona, Edi\-fici C, 08193 Bellaterra, Barcelona,
Catalonia, Spain} \email{jllibre@mat.uab.cat}

\address{$^2$ Departament de Matem\`atica Aplicada I, Universitat
Polit\`ecnica de Cata\-lunya, (EPSEB), Av. Doctor Mara\~{n}\'on, 44--50,
08028 Barcelona, Spain} \email{chara.pantazi@upc.edu}



\begin{document}

\begin{abstract}
We study the maximum number of limit cycles that can bifurcate from
a degenerate center of a cubic homogeneous polynomial differential
system. Using the averaging method of second order and perturbing
inside the class of all cubic polynomial differential systems we
prove that at most three limit cycles can bifurcate from the
degenerate center. As far as we know this is the first time that a
complete study up to second order in the small parameter of the
perturbation is done for studying the limit cycles which bifurcate
from the periodic orbits surrounding a degenerate center (a center
whose linear part is identically zero) having neither a Hamiltonian
first integral nor a rational one. This study needs many
computations, which have been verified with the help of the
algebraic manipulator Maple.
\end{abstract}

\maketitle

\section{Introduction}
Hilbert in \cite{H} asked for the maximum number of limit cycles
which real polynomial differential systems in the plane of a given
degree can have. This is actually the well known {\it 16th Hilbert
Problem}, see for example the surveys  \cite{Yu,Li}  and references
therein. Recall that a {\it limit cycle} of a planar polynomial
differential system is a periodic orbit of the system isolated in
the set of all periodic orbits of the system.

\medskip

Poincar\'e in \cite{point1} was the first to introduce the notion of
a center  for a vector field  defined on the real plane. So
according to Poincar\'e  a {\it center} is a singular point
surrounded by a neighborhood filled of periodic orbits with the
unique exception of the singular point.

\medskip
Consider the polynomial differential system
\begin{equation}
\dot{x}=P(x,y), \qquad \dot{y}=Q(x,y),
\label{first}
\end{equation}
and as usually  we denote by $ \dot{}={d}/{dt} $. Assume that system
\eqref{first}  has a center  located at the origin. Then after a
linear change of variables and a possible scaling of time system
\eqref{first} can be written in one of the following forms
$$
\begin{array}{lll}
\begin{array}{ll}
(A) &
%\left\{
\begin{array}{l}
\dot{x}=-y+F_1(x,y),\\
\dot{y}=x+F_2(x,y),
\end{array}
%\right.
\end{array}
&
\begin{array}{ll}
(B) &
%\left\{
\begin{array}{l}
\dot{x}=y+F_1(x,y),\\
\dot{y}=F_2(x,y),
\end{array}
%\right.
\qquad
\end{array}
&
\begin{array}{ll}
(C) &
%\left\{
\begin{array}{l}
\dot{x}=F_1(x,y),\\
\dot{y}=F_2(x,y),
\end{array}
%\right.
\end{array}
\end{array}
$$
with $F_1$ and $F_2$ polynomials  without constant and linear terms.
When system \eqref{first} can be written into the form (A) we say
that the center is of {\it linear type}. When system \eqref{first}
can take the form  (B)  the center is {\it nilpotent}, and when
system \eqref{first} can be transformed into the form  (C) the
center is {\it degenerate}.

\medskip

Due to the difficulty of this problem mathematicians have consider
simpler versions.  Thus Arnold \cite{Arnold} considered the {\it
weakened 16th Hilbert Problem}, which consists in determining an
upper bound for the number of limit cycles which can bifurcate  from
the periodic orbits of a polynomial Hamiltonian center when it is
perturbed inside a class of polynomial differential systems, see for
instance \cite{ChrisLi} and the hundred of references quoted
therein. It is known that in a neighborhood of a center always there
is a first integral, see \cite{MS}. When this first integral  is not
polynomial the computations become more difficult. Moreover, if the
center is degenerate the computations become even harder.

\medskip

In the literature we can basically find the following methods for
studying the limit cycles that bifurcate from a center:
\begin{itemize}
\item The method that uses the Poincar\'e return map, like the
articles \cite{BP, chico}.

\item The one that uses the Abelian integrals  or Melnikov integrals
(note that for systems in the plane the two notions are equivalent),
see for example section 5 of Chapter 6 of \cite{ArnoldIly} and
section 6 of Chapter 4 of \cite{GH}.

\item The one that uses the inverse integrating factor, see
\cite{GLV1,GLV2,GLV3,GLV4}.

\item The averaging theory \cite{Buica,GGLD,llibre, Sanders, Verhulst}.
\end{itemize}
The first two methods provide information about the number of limit
cycles whereas the last two methods additionally give the shape of
the bifurcated limit cycle up to any order in the perturbation
parameter.

\medskip

Almost all the papers studying how many limit cycles can bifurcate
from the periodic orbits of a center, work with centers of linear
type. There are very few papers studying this problem for nilpotent
or degenerate centers. In fact, for degenerate centers as far as we
known the bifurcation of limit cycles from the periodic orbits of a
degenerate center only have been studying completely using formulas
of first order in the small parameter of the perturbation. Here we
will provide a complete study of this problem using formulas of
second order, and as it occurs with the formulas of second order
applied to linear centers that they provide  in general more limit
cycles than the formulas of first order, the same occurs for the
formulas of second order applied to degenerate centers. Of course,
the computations from first order to second order increases almost
exponentially.

\medskip

This paper deals with the weakened 16th Hilbert's problem but
perturbing  non--Hamiltonian  degenerate centers using the technique
of the averaging method of second order, see \cite{GGLD}, and
Section \ref{saveraging} for  a summary of the results that we need
here.

\medskip

Since we want to study the perturbation of a degenerate center with
averaging of second order, from the homogeneous centers the first
ones that are degenerate, are the cubic homogeneous centers, see for
instance \cite{CL}. In this class in \cite{LLTT} the authors studied
the perturbation of the following cubic homogeneous center
\begin{equation}\label{1}
\dot{x}=-y(3x^2+y^2),\qquad \dot{y}=x(x^2-y^2),
\end{equation}
inside the class of all cubic polynomial differential systems, using
averaging theory of first order. Here we study this problem but
using averaging theory of second order.

\medskip

System \eqref{1} has a global center at the origin (i.e. all the
orbits contained in $\R^2\setminus\{(0,0)\}$ are periodic), and it
admits the non--rational first integral
$$
H(x,y)=(x^2+y^2)\exp\left(-\frac{2x^2}{x^2+y^2}\right).
$$
The limit cycles bifurcating from the periodic orbits of the global
center \eqref{1} have already been studied  in the following two
results, see \cite{LLTT} and \cite{Buica0}, respectively.
\begin{theorem}
We deal with differential system \eqref{1}. Then the polynomial
differential system
$$
\begin{array}{ccl}
\dot{x}&=&-y(3x^2+y^2)+\varepsilon \left(
\sum\limits_{0\leq i+j\leq 3}a_{ij}x^iy^{j}\right),\\
\dot{y}&=&x(x^2-y^2)+\varepsilon\left(
\sum\limits_{0\leq i+j\leq 3}b_{ij}x^iy^{j}\right),
\end{array}
$$
has at most one limit cycle bifurcating from the periodic orbits of
the  center of system \eqref{1} using averaging theory of first
order. Moreover, there are examples with 1 and 0 limit cycles.
\end{theorem}

\begin{proposition}
We consider the homogeneous polynomial differential system
\eqref{1}. Let $P_i(x,y)$ and $Q_i(x,y)$ for $i=1,2$ be polynomials
of degree at most 3. Then for convenient polynomials $P_i$ and
$Q_i$, the polynomial differential system
$$
\begin{array}{ccl}
\dot{x}&=&-y(3x^2+y^2)+\varepsilon P_1(x,y)+\varepsilon^2 P_2(x,y),\\
\dot{y}&=&x(x^2-y^2)+\varepsilon Q_1(x,y)+\varepsilon^2 Q_2(x,y),
\end{array}
$$
has at first order averaging one limit cycle, and at second order
averaging two limit cycles bifurcating from the periodic solutions
of the global center \eqref{1}.
\end{proposition}

Our main result is the following one and it do by first time the
complete study of the averaging method of second order for a
degenerate center having neither a Hamiltonian first integral nor a
rational one.

\begin{theorem}\label{main}
We consider the cubic homogeneous differential system \eqref{1}.
Then the  perturbation of system \eqref{1} inside the class of all
cubic polynomial systems
\begin{equation}
\begin{array}{ccl}
\dot{x}&=&-y(3x^2+y^2)+\varepsilon \left(
\sum\limits_{0\leq i+j\leq 3}a_{ij}x^iy^{j}\right)+\varepsilon^2 \left(
\sum\limits_{0\leq i+j\leq 3}b_{ij}x^iy^{j}\right),\\
\dot{y}&=&x(x^2-y^2)+\varepsilon\left(
\sum\limits_{0\leq i+j\leq 3}c_{ij}x^iy^{j}\right)+\varepsilon^2 \left(
\sum\limits_{0\leq i+j\leq 3}d_{ij}x^iy^{j}\right),
\end{array}
\label{perturbed}
\end{equation}
has at most three limit cycles  bifurcating from the  periodic
orbits of the center of system \eqref{1} using averaging theory of
second order. Moreover, there are examples with 3, 2, 1 and 0 limit
cycles.
\end{theorem}

The paper is organized as follows: In section \ref{saveraging} we
present a summary of the averaging method of second order following
\cite{GGLD}. Next in section \ref{smain} we provide the proof of our
main Theorem \ref{main}. In section \ref{sexamples} we provide three
examples, of systems \eqref{perturbed} with 0, 1, 2 and 3 limit
cycles bifurcating from the degenerate center. At the end we present
the Appendices A, B and C.

\section{The averaging method of second order}\label{saveraging}

In this section we present the averaging method of second order
following \cite{GGLD}. In that paper the averaging theory for
differential equations of one variable is done up to any order in
the small parameter of the perturbation. We consider the analytic
differential equation
\begin{equation}
\dfrac{dr}{d\theta}=G_0(\theta, r)+\sum\limits_{k\geq
1}\varepsilon^kG_k(\theta, r),\label{eq0}
\end{equation}
with $r\in\R$, $\theta\in \aS^1$ and $\varepsilon\in(-\varepsilon_0,
\varepsilon_0)$ with $\varepsilon_0$ a small positive real value,
and the functions $G_k(\theta, r)$ are $2\pi$--periodic in the
variable $\theta.$ Note that for $\varepsilon=0$ system \eqref{eq0}
is unperturbed. Let  $r_s(\theta, r_0)$ be the solution of  system
\eqref{eq0} with $\varepsilon=0$ satisfying $r_s(0,r_0)=r_0$ and
$r_s(\theta, r_0)$ is $2\pi$ periodic for $r_0\in\mathcal{I}$ with
$\mathcal{I}$ a real open interval. We are interested in the limit
cycles of equation \eqref{eq0} which bifurcate from the periodic
orbits of the unperturbed system with initial condition
$r_0\in\mathcal{I}.$ So, we define by $r_{\varepsilon}(\theta,r_0)$
the solution of equation \eqref{eq0} satisfying
$r_{\varepsilon}(0,r_0)=r_0.$

\medskip

In what follows we denote by $u=u(\theta, r_0)$ the solution of the
variational equation
$$
\dfrac{\partial u}{\partial \theta}=\dfrac{\partial G_0}{\partial
r}(\theta, r_s(\theta, r_0))u,
$$
satisfying $u(0, r_0)=1.$

\medskip
We define
\begin{equation}\label{formulas}
\begin{array}{ccl}
u_1(\theta, r_0)&=&\displaystyle{\int\limits_{0}\limits^\theta
\dfrac{ G_1(\phi,\ r_s(\phi, r_0))}{u(\phi, r_0)}d\phi}
=\displaystyle{\int\limits_{0}\limits^\theta \dfrac{ G_1(w,\ r_s(w,
r_0))}{u(w, r_0)}d w},\\
G_{10}(r_0)&=&\displaystyle{\int\limits_{0}\limits^{2\pi} \dfrac{
G_1(\theta, \ r_s(\theta, r_0))}{u(\theta, r_0)}d\theta},\\
G_{20}(r_0)&=&\displaystyle{\int\limits_{0}\limits^{2\pi}
\left(\dfrac{ G_2(\theta, \ r_s(\theta, r_0))}{u(\theta, r_0)}
+\dfrac{\partial G_1}{\partial r}
(\theta, \ r_s(\theta, r_0))\ u_1(\theta,r_0)\right.}\\
&&\displaystyle{\qquad \qquad +\left.\dfrac{1}{2}\dfrac{\partial^2
G_0}{\partial r^2}(\theta, \ r_s(\theta, r_0))\ u_1(\theta,
r_0)^2\right)d\theta}.
\end{array}
\end{equation}

In statement (b) of Corollary 5 of \cite{GGLD} it is proved the
following result.

\begin{theorem}
Assume that the solution $r_s(\theta, r_0)$ of the unperturbed
equation \eqref{eq0} such that $r_s(0, r_0)=r_0$ is $2\pi$--periodic
for $r_0\in\mathcal{I}$ with $\mathcal{I}$ a real open interval. If
$G_{10}(r_0)$ is identically zero in $\mathcal{I}$  and
$G_{20}(r_0)$ is not identically zero in $\mathcal{I}$, then for
each simple zero $r^*\in\mathcal{I}$ of $G_{20}(r_0)=0$ there exists
a periodic solution $r_\varepsilon(\theta, r_0)$ of \eqref{eq0} such
that $r_\varepsilon(0, r_0)\rightarrow r^*$ when
$\varepsilon\rightarrow 0.$
\end{theorem}

\section{Proof of Theorem \ref{main}}\label{smain}

System \eqref{1} in polar coordinates becomes
\begin{equation}
\dot{r}=-2r^3\cos\theta \sin\theta, \qquad \dot{\theta}=r^2,
\label{1polar}
\end{equation}
or equivalently,
$$
\dfrac{dr}{d\theta}=-2r\cos\theta \sin\theta,
$$
and it has the solution $r_s(\theta, r_0)=r_0 \exp(-\sin ^2 \theta)$
satisfying that $r_s(0, r_0)=r_0.$

\medskip

Now we perturb system \eqref{1} inside the class of all cubic
polynomial differential systems as in \eqref{perturbed}. System
\eqref{perturbed} in polar coordinates give rise to the differential
equation
$$
{\frac {d\,r}{d\,\theta}}=G_0(\theta, r)+\varepsilon G_1(\theta,
r)+\varepsilon^2 G_2(\theta,
r)+O(\varepsilon^3),
$$
with
$$
\begin{array}{l}
G_0(\theta, r)=-2\,r\cos  \theta  \sin  \theta ,\\
\\
G_1(\theta,
r)=g_{1,1}(\theta)\,\dfrac{1}{r^2}+g_{1,2}(\theta)\,\dfrac{1}
{r}+g_{1,3}(\theta)\,+g_{1,4}(\theta)\,r,\\
\\
G_2(\theta,r)=g_{2,1}(\theta)\,\dfrac{1}{r^5}+g_{2,2}(\theta)\,
\dfrac{1}{r^4}+g_{2,3}(\theta)\,\dfrac{1}{r^3}
+g_{2,4}(\theta)\,\dfrac{1}{r^2}+g_{2,5}(\theta)\,\dfrac{1}{r}+
g_{2,6}(\theta)\,+g_{2,7}(\theta)\,r,
\end{array}
$$
where the expressions of the coefficients $g_{1,i}(\theta)\,$ for
$i=1,2,3,4$ and $g_{2,j}(\theta)\,$ for $j=1,2,\cdots,7$ are given
in the Appendix A.

\medskip

Additionally, we consider the variational equation
$$
\dfrac{\partial u}{\partial \theta}=\dfrac{\partial G_0}{\partial
r}(\theta, r_s(\theta, r_0)),
$$
and its solution $u(\theta, r_0)$ satisfying $u(0, r_0)=1$, namely
$u_s(\theta)=\exp(-\sin^2 \theta).$

\medskip
We define
\begin{equation}{\label{integrals}}
\begin{array}{l}
I_1=\displaystyle{\int\limits_{0}\limits^{2\pi}\exp(2\sin^2\theta)\
\cos^4\theta\ d\theta}=3.572403292...,\\
I_2=\displaystyle{\int\limits_{0}\limits^{2\pi}\exp(2\sin^2\theta)\
\cos^2\theta\ d\theta}=5.985557563...,\\
I_3=\displaystyle{\int\limits_{0}\limits^{2\pi}\exp(2\sin^2\theta) \
d\theta}=21.62373221...\\
\end{array}
\end{equation}

\begin{lemma}
Consider $I_1, I_2, I_3$ defined in \eqref{integrals}. Then for
$a_{10}=-(I_3-2I_1+I_2)/({2I_1-I_2})c_{01}$  and $a_{30}=
-2c_{03}-c_{21}$ we have that the function $G_{10}(r_0)$ defined in
\eqref{formulas} is identically zero.
\end{lemma}

\begin{proof}
We have
$$
\dfrac{G_1(\theta, r_s(\theta,r_0))}{u(\theta,r_0)}=A(\theta)\,
\dfrac{1}{r_0^2}+B(\theta)\,\dfrac{1}{r_0}+C(\theta)\,+D(\theta)\,r_0,
$$
with
$$
\begin{array}{l}
A(\theta)\,=\left[\sin \theta  \left( 2\,  \cos^2 \theta +1 \right) c_{00}

 +\cos \theta  \left( -1+
2\,  \cos^2 \theta  \right) a_{00}\right]{{\rm e}^{3
\,  \sin^2 \theta }},\\
\\

B(\theta)\,=\left[\cos \theta \sin \theta  \left( -1+2\,
  \cos^2 \theta  \right) a_{01}
 +
  \cos^2 \theta  \left( -1+2\,  \cos^2 \theta  \right)a_{{1
0}}\right.\\
\left.- \left( -  \cos^2 \theta +2\,  \cos^4 \theta -1 \right)c_{01}
 +\cos \theta \sin \theta  \left( 2\,  \cos^2 \theta +1
 \right) c_{10}\right]{{\rm e}^{2
\,  \sin^2 \theta }},
\\
\\
C(\theta)\,=\left[-
\cos \theta  \left( -3\,  \cos^2 \theta +1+2\,  \cos^4 \theta
 \right)a_{02}
 +  \cos^3 \theta  \left( -
1+2\,  \cos^2 \theta  \right)a_{20}\right.
\\
 +  \cos^2 \theta \sin \theta  \left( -1+2\,
  \cos^2 \theta  \right) a_{11}
 +  \sin^3 \theta  \left( 2\,  \cos^2 \theta +1 \right)c_{{0
2}}
\\
\left.+  \cos^2 \theta \sin \theta  \left( 2\,  \cos^2 \theta +1 \right)c_{20}
-\cos \theta  \left( -
  \cos^2 \theta +2\,  \cos^4 \theta -1 \right)c_{11}\right]{{\rm e}^{  \sin^2 \theta }},\\
\\
D(\theta)\,=
\cos \theta   \sin^3 \theta  \left( -1+2\, \cos^2 \theta   \right)a_{03}
+  \cos^4 \theta
 \left( -1+2\,  \cos^2 \theta  \right)a_{30}\\
  +
  \cos^3 \theta \sin \theta  \left( -1+2\,  \cos^2 \theta   \right)a_{21}
-  \cos^2 \theta
 \left( -3\,  \cos^2 \theta +1+2\,
  \cos^4 \theta  \right)a_{12}\\
 +
 \left( 1-3\,  \cos^4 \theta +2\,
 \cos^6 \theta   \right)c_{03}
  +
  \cos^3 \theta \sin \theta  \left( 2\,  \cos^2 \theta +1
 \right)c_{30}\\
 -  \cos^2 \theta
 \left( -  \cos^2 \theta +2\,  \cos^4 \theta -1 \right)c_{21}
+\cos
 \theta   \sin^3 \theta
 \left( 2\,  \cos^2 \theta +1 \right)c_{12}.
\end{array}
$$

Now
$$
G_{10}(r_0)=\displaystyle{\int\limits_{0}\limits^{2\pi}\dfrac{G_1(\theta,
r_s(\theta,r_0))}{u(\theta,r_0)}}d\theta=
\displaystyle{\int\limits_{0}\limits^{2\pi}\left(A(\theta)\,\dfrac{1}
{r_0^2}+B(\theta)\,\dfrac{1}{r_0}+C(\theta)\,+D(\theta)\,r_0\right)d\theta,}
$$
and considering the change of coordinates $\theta=\phi+\pi$ in the
interval $[0, 2\pi]$ and the symmetries
\begin{equation}\label{simmetry}
\sin(\theta+\pi)=-\sin\theta, \qquad
\cos(\theta+\pi)=-\cos\theta,
\end{equation}
we have that
$$
\displaystyle{\int\limits_{0}\limits^{2\pi}A(\theta) d\theta}=\displaystyle{
\int\limits_{-\pi}\limits^{\pi}A(\phi) d\phi}=0,
\qquad \displaystyle{\int\limits_{0}\limits^{2\pi}C(\theta) d\theta}=
\displaystyle{\int\limits_{-\pi}\limits^{\pi}C(\phi) d\phi}=0.
$$
So we have
$$
\begin{array}{ccl}
G_{10}(r_0)&=&\displaystyle{\int\limits_{0}\limits^{2\pi}  \dfrac{B(\theta)\,}
{r_0} d\theta}+\displaystyle{\int\limits_{0}\limits^{2\pi} D(\theta)\, r_0 d\theta}\\
&=&\left[(2a_{10}-2c_{01})I_1+(c_{01}-a_{10})I_2+c_{01}I_3\right]\dfrac{1}{r_0}
+\dfrac{\pi}{2}\left(2c_{03}+a_{30}+c_{21}\right)r_0,
\end{array}
$$
\noindent and therefore  $G_{10}\equiv 0$ if
$$
a_{10}=-\dfrac{I_3-2I_1+I_2}{2I_1-I_2}c_{01}=-17.65322447\ldots
c_{01}, \qquad a_{30}= -2c_{03}-c_{21}.
$$
This completes the proof of the lemma.
\end{proof}

Now we have
$$
\dfrac{G_2(\theta,
r_s(\theta,r_0))}{u(\theta,r_0)}=A_5(\theta)\,\dfrac{1}{r_0^5}+
A_4(\theta)\,\dfrac{1}{r_0^4}+A_3(\theta)\,\dfrac{1}{r_0^3}
+A_2(\theta)\,\dfrac{1}{r_0^2}+A_1(\theta)\,\dfrac{1}{r_0}+
A_0(\theta)\,+\tilde{A}_1(\theta)\,r_0,
$$
with
$A_5(\theta)\,,A_4(\theta)\,,A_3(\theta)\,,A_2(\theta)\,,
A_1(\theta)\,,A_0(\theta)\,,\tilde{A}_1(\theta)\,$
are given in the Appendix B.

\medskip

We note that
$$
\displaystyle{\int\limits_{0}\limits^{2\pi}A_4(\theta)\, d\theta
}=0, \qquad \displaystyle{\int\limits_{0}\limits^{2\pi}A_2(\theta)\,
d\theta }=0, \qquad
\displaystyle{\int\limits_{0}\limits^{2\pi}A_0(\theta)\, d\theta
}=0,
$$
because of the symmetries \eqref{simmetry}.
So we have
\begin{equation}\label{ints}
\displaystyle{\int\limits_{0}\limits^{2\pi}\dfrac{G_2(\theta,
r_s(\theta,r_0))}{u(\theta,r_0)}}d\theta
=\dfrac{1}{r_0^5}\displaystyle{\int\limits_{0}\limits^{2\pi}A_5(\theta)\,d\theta}
+\dfrac{1}{r_0^3}\displaystyle{\int\limits_{0}\limits^{2\pi}A_3(\theta)\,d\theta
}+\dfrac{1}{r_0}\displaystyle{\int\limits_{0}\limits^{2\pi}A_1(\theta)\,d\theta}
+\left(\displaystyle{\int\limits_{0}\limits^{2\pi}\tilde{A}_1(\theta)\,d\theta}\right)
r_0,
\end{equation}
and we recall that the expressions of
$A_5(\theta)\,,A_3(\theta)\,,A_1(\theta)\,,\tilde{A}_1(\theta)\,$
are given in the Appendix B. We have
$$
\begin{array}{ll}
\displaystyle{\int\limits_{0}\limits^{2\pi}A_5(\theta)\, d\theta }&=\displaystyle{ \left( -4\int _{0}^{2\,\pi }\,{{\rm e}^{6\,
  \sin^2  \theta }}  \cos^4
\theta  {d\theta }+2\int _{0}^{2\,\pi }\,{{\rm e}^{6
\,  \sin^2 \theta }}  \cos^2 \theta {d\theta }+\int _{0}^{2\,\pi }\!{{\rm e}^{6\,
  \sin^2 \theta }}{d\theta } \right)a_{{00}}c_{{00}}}\\
 &=665.2264930\ldots a_{{00}}c_{{00}},\\

 \displaystyle{\int\limits_{0}\limits^{2\pi}A_3(\theta)\, d\theta }&=  239.0000390\ldots\,a_{01}c_{01}+ 97.83745135\ldots\,a_{00}c_{02}
 +97.83745135\ldots\,a_{02}c_{00}\\
 &- 257.2692783\ldots\,c_{01}c_{10}+
 12.93483815\ldots\,a_{00}c_{20}- 7.99641945\ldots\,a_{00}a_{11}\\
 &+12.93483815\ldots\,a_{20}c_{00}- 28.92767705\ldots\,c_{00}c_{11},\\

 \displaystyle{\int\limits_{0}\limits^{2\pi}A_1(\theta)\, d\theta }&=1.159249021\ldots\,b_{10}+ 20.46448319\ldots\,d_{{01}}+ 2.318498043\ldots\,a_{{11}
}c_{11}- 4.73165232\ldots\,c_{02}c_{11}\\
&+ 2.318498043\ldots\,a_{20}c_{{0
,2}}+ 2.318498045\ldots\,a_{02}c_{20}- 3.761715750\ldots\,c_{11}c_{20}\\
&+ 14.47892563\ldots\,a_{02}c_{02}- 1.253905250\ldots\,a_{02}a_{11}+
 0.189312456\ldots\,a_{20}c_{20}\\
 &+ 0.094656226\ldots\,a_{11}a_{20}+
 0.647510463\ldots\,a_{21}c_{01}- 5.110277230\ldots\,c_{03}c_{10}\\
 &+
 2.318498043\ldots\,a_{12}c_{10}+ 2.22384182\ldots\,a_{01}c_{21}+
 36.6143963\ldots\,a_{03}c_{01}\\
& - 3.95102821\ldots\,c_{10}c_{21}-
 1.25390525\ldots\,a_{01}a_{12}- 45.6606188\ldots\,c_{01}c_{12}\\
 &-
 7.1036910\ldots\,c_{01}c_{30}+ 14.28961317\ldots\,a_{01}c_{03},\\

\displaystyle{\int\limits_{0}\limits^{2\pi}\tilde{A}_1(\theta)\, d\theta }&= - 0.3926990817\ldots\,c_{30}c_{21}+ 0.1963495408\ldots\,c_{30}a_{12}+
 2.159844949\ldots\,a_{03}c_{03}\\
&- 0.1963495408\ldots\,a_{03}a_{12}-
 0.7853981634\ldots\,c_{12}c_{21}+ 0.3926990817\ldots\,a_{03}c_{21}\\
 &-1.178097245\ldots\,c_{03}c_{12}+ 0.3926990817\ldots\,c_{12}a_{12}+
 0.9817477042\ldots\,c_{03}c_{30}\\
& + 1.570796327\ldots\,d_{{21}}+ 3.141592654\ldots
\,d_{03}+ 1.570796327\ldots\,b_{30}.\\
\\
%K_1int&=&1/2\,d_{{2,1}}\pi +1/2\,b_{30}\pi +d_{03}\pi -1/4\,c_{12}c_{{
%2,1}}\pi -1/8\,c_{21}c_{30}\pi -1/16\,a_{03}a_{12}\pi +{
%\frac {11}{16}}\,c_{03}a_{03}\pi\\
%&&+1/8\,a_{03}c_{21}\pi +1/
%8\,a_{12}c_{12}\pi +1/16\,a_{12}c_{30}\pi -3/8\,c_{03}c
%_{{1,2}}\pi +{\frac {5}{16}}\,c_{03}c_{30}\pi
\end{array}
$$
\begin{remark} \label{remark1}
(a) Looking at the expressions of $A_3, A_1,\tilde{A}_1$ in the
Appendix B we can have the exact definition for the numerical
coefficients which appear in the previous integrals. Thus for
instance
$$
239.0000390\cdots=-\displaystyle{\int\limits_{0}\limits^{2\pi}}{\frac {{{\rm e}^{4\,  \sin^2  \theta}}
 \left( \cos^2  \theta -1 \right)    \left( 2\,{ I_1}-4{ I_3}\,  \cos^4 \theta
   +2{ I_3}\,  \cos^2 \theta -{ I_2} \right) }{2\,{ I_1}-{ I_2}}} d \theta,
$$
and $I_1,I_2,I_3$ satisfying relations \eqref{integrals}.

\medskip

(b) All the computations of this paper have been verified with the
algebraic manipulator Maple \cite{maple}.
\end{remark}

We additionally have
$$
\dfrac{\partial G_1}{\partial r}(\theta, \ r_s(\theta,
r_0))=B_0(\theta)\,+\dfrac{B_1(\theta)\,}{r_0^2}+\dfrac{B_2(\theta)\,}{r_0^3},
$$
with
$$
\begin{array}{ll}
B_0(\theta)\,=&
\cos^2 \theta \left( 1+ 2\,
  \cos^2 \theta
 - 4\,  \cos^4 \theta \right)\, c_{21}\\
 &+
\left( - 2\,
  \cos^6 \theta + 1- \cos^4 \theta  \right) c_{03}\\
 &+ \cos \theta \left(  \,
\sin \theta   \cos^2 \theta - 2\,  \cos^4 \theta  \sin \theta + \, \sin \theta
 \right) c_{12}\\
 &+ \cos^3 \theta \left(  \,\sin \theta   + 2\,  \cos^2 \theta \sin \theta  \right) c_{30}
 \\
 &+\sin \theta \,\cos \theta
 \left(  3\,   \cos^2 \theta - 1  - 2\,  \cos^4 \theta   \right) a_{03}\\
 &+ \cos^2 \theta\left( -1 \,   + 3\,  \cos^2 \theta - 2\,  \cos^4 \theta  \right) a_{12}\\
 &+  \,\sin \theta   \cos^3 \theta\left( -1 + 2\,  \cos^2 \theta
 \right) a_{21},\\
 \\
 B_1(\theta)\,=&%+\dfrac
 { \left( - 1- 18.65322447...\,\cos^2\theta
 +
 37.30644894...\,   \cos^4 \theta  \right) {{\rm e}^{  2\sin^2 \theta }} c_{01}}%{r_0^2}
\\
%\dfrac
&{\,
-\sin \theta \cos \theta   \left(2\,   \cos^2 \theta  +1\right) {{\rm e}^{  2\sin^2 \theta }} c_{10}}%{r_0^2}
\\
&+%\dfrac{
\sin \theta \, \cos \theta \left( - 2\,   \cos^2 \theta  + 1 \right) {{\rm e}^{  2\sin^2 \theta }} a_{01},%{{{\it r_0}}^{2}},
 \\
 \\
B_2(\theta)\,=&%+{\dfrac
{-2\,\sin
 \theta \left(  2\,   \cos^2 \theta  + 1  \right) c_{00}+ 2\left( - 2\, + \,  {{\rm e}^{  3\sin^2 \theta }} \cos
 \theta  \right) a_{00}}.
 %{{{\it r_0}}^{3}}}
\end{array}
$$

Now we have
$$
\dfrac{ G_1(w,\ r_s(w, r_0))}{u(w, r_0)}=\dfrac{C_2(w)\,}{r_0^2}+\dfrac{C_1(w)\,}{r_0}+C_0(w)\,+{\tilde C}_1(w)\, r_0,
$$
with
$$
\begin{array}{ll}
C_2(w)\,=&{{\rm e}^{3 \sin^2  w }}\left[\cos  w \left( 2\,  \cos^2  w -1 \right) a_{00}
+  \sin  w   \left( 1+2\,  \cos^2  w
 \right) c_{00}\right],\\
 \\
 C_1(w)\,=&\Big[ \cos  w \sin  w\left( -1  +2\,
    \cos^2  w     \right) a_{01}\\

 &+ \left(1+ \dfrac {I_3}{2I_1-I_2}   (\cos^2 w -2\cos^4 w )  \right) c_{01}\\
 &+ \cos  w \sin  w\ \left(  1  +2\, \cos^2  w   \right) c_{10}\Big] {{\rm e}^{
 2\sin^2  w   }},
\\
\\
C_0(w)\,=& \left[ \cos^3  w   \left( 2\,  \cos^2  w -1 \right) a_{20}
 -\cos  w \left( -3\,
  \cos^2  w +2\,  \cos^4  w  +1 \right) a_{02}\right.\\
 &+\sin  w    \cos^2  w  \left( 2\,  \cos^2  w -1 \right) a_{11}
 +\sin  w    \cos^2  w  \left( 1+2\,  \cos^2  w  \right) c_{20}\\
 &+c_{02}  \sin^3  w   \left( 1+2\,  \cos^2  w
 \right)\left.
  -\cos  w \left( 2\,  \cos^4  w -1-  \cos^2  w  \right) c_{11}\right]{{\rm e}^{  \sin^2  w  }},
\\
\\
{\tilde{C}_1}(w)\,=&- \left(   \cos^4  w  +2\,  \cos^6  w -1 \right) c_{03}
 +  \sin^3  w  \cos  w \left( 1+2\,
  \cos^2  w  \right)c_{12}\\
 &-  \cos^2  w  \left( -1-2\,  \cos^2  w +4\,  \cos^4  w
 \right)c_{21} +  \cos^3  w  \sin  w   \left( 1+2\,  \cos^2  w  \right)c_{30}\\
&+ \sin^3  w  \cos  w \left( 2\, \cos^2  w -1 \right)a_{03}  + \sin^2  w
  \cos^2  w  \left( 2\,  \cos^2  w -1 \right) a_{12} \\
 &+ \cos^3  w  \sin  w   \left( 2\,
  \cos^2  w -1 \right)a_{21}.
\end{array}
$$

\noindent Additionally, from \eqref{formulas} we obtain
$$
u_1(\theta, r_0)=\dfrac{1}{r_0^2}\displaystyle{\int\limits_{0}\limits^\theta C_2(w)\,}dw
+\dfrac{1}{r_0}\displaystyle{\int\limits_{0}\limits^\theta C_1(w)\,}dw+ \displaystyle{\int\limits_{0}\limits^\theta}C_0(w)\,dw
+r_0\displaystyle{\int\limits_{0}\limits^\theta {\tilde C}_1(w)\, }dw,
$$
and so
\begin{equation}\label{ss}
\dfrac{\partial G_1}{\partial r}
(\theta, \ r_s(\theta, r_0))\ u_1(\theta,r_0)=s_{5}(\theta)\, \dfrac{1}{r_0^5}+s_{4}(\theta)\, \dfrac{1}{r_0^4}+s_3(\theta)\, \dfrac{1}{r_0^3}
+s_2(\theta)\, \dfrac{1}{r_0^2}+s_1(\theta)\, \dfrac{1}{r_0}+s_0(\theta)\, +\tilde{s}_1(\theta)\,  r_0,
\end{equation}
and the explicit expressions of $s_i(\theta)\, $ for $i=0,1,\cdots,5$ and $\tilde{s}_1(\theta)\, $ are given in the Appendix C.

\medskip

Since $\dfrac{\partial^2 G_0}{\partial r^2}=0$ from \eqref{formulas}
we have that
$$
G_{20}(r_0)=\displaystyle{\int\limits_{0}\limits^{2\pi}
\left(\dfrac{ G_2(\theta, \ r_s(\theta, r_0))}{u(\theta, r_0)}
+\dfrac{\partial G_1}{\partial r}
(\theta, \ r_s(\theta, r_0))\ u_1(\theta,r_0)\right)d\theta},\\
$$
and we obtain
$$
r_0^5{G_{20}(r_0)}=v_6 r_0^6+v_4 r_0^4+v_2 r_0^2+v_0,
$$
with
$$
\begin{array}{ll}
v_6=& - 0.3926990800\cdots\,c_{21}c_{30}+ 0.1963495397\cdots\,a_{12}c_{
{30}}+ 2.159844949\cdots\,a_{03}c_{03}\\
&- 0.1963495365\cdots\,a_{03}a_{{1
2}} - 0.7853981634\cdots\,c_{12}c_{21}+ 0.3926990817\cdots\,a_{03}c_{{21}
}
\\
&- 1.178097245\cdots\,c_{03}c_{12}+ 0.3926990817\cdots\,a_{12}c_{12}+
 0.9817477042\cdots\,c_{03}c_{30}\\

 &+ 1.570796327\cdots\,d_{{21}}+ 3.141592654\cdots
\,d_{03}+ 1.570796327\cdots\,b_{30},
\\
\\
v_4=&-
 3.155691751\cdots\,a_{21}c_{01}+ 6.612510180\cdots\,c_{03}c_{10}+
 1.786201647\cdots\,a_{12}c_{10}\\
 &+ 2.413154277\cdots\,a_{01}c_{21}+
 38.88726613\cdots\,a_{03}c_{01}- 1.253905255\cdots\,c_{10}c_{21}\\
 &-
 1.206577137\cdots\,a_{01}a_{12}- 68.30745733\cdots\,c_{01}c_{12}-38.88726635\cdots\,c_{01}c_{30}\\

 &+ 13.27234849\cdots\,a_{01}c_{03}+
 2.318498043\cdots\,a_{11}c_{11}- 4.73165232\cdots\,c_{02}c_{11}\\

& +2.318498043\cdots\,a_{20}c_{02}+ 2.318498045\cdots\,a_{02}c_{20}-
 3.761715750\cdots\,c_{11}c_{20}\\
 &+ 14.47892563\cdots\,a_{02}c_{02}-1.253905250\cdots\,a_{02}a_{11}+ 0.189312456\cdots\,a_{20}c_{20}\\
 &+
 0.094656226\cdots\,a_{11}a_{20}+ 1.159249021\cdots\,b_{10}+ 20.46448319\cdots
\,d_{{01}},\\
\\
v_2=
&95.95703341\cdots\,a_{00}c_{02}+105.7377762\cdots\,a_{02}c_{00}+ 239.0000390\cdots\,a_{01}c_{01}\\

&-7.649140220\cdots\,a_{00}a_{11}+ 16.68739736\cdots\,a_{00}c_{20}-110.9164314\cdots\,c_{00}c_{11}\\
&- 257.2692783\cdots\,c_{01}c_{10}-0.000001\cdots\,{c_{01}}^{2}
- 31.79348852\cdots\,a_{20}c_{00},\\
\\
v_0=& 665.2264933\cdots\,a_{00}c_{00}.
\end{array}
$$
We have that the coefficients  $v_6,v_4,v_2,v_0$ are independent
because $d_{03}$ only appears in $v_6$, $b_{10}$ only appears in
$v_4$, $a_{00}c_{02}$ only appears in $v_2$, and $a_{00}c_{00}$ only
appears in $v_0.$

\medskip

Now we are going to use Descartes Theorem:

\begin{theorem}[Descartes Theorem]\label{Descartes}
Consider the real polynomial $p(x)=a_{i_1}x^{i_1}+a_{i_2}x^{i_2}
\cdots+a_{i_r}x^{i_r}$ with $0\leq i_1< i_2< \cdots<i_r$ and
$a_{i_j}\neq 0$ real constants for $j\in\{1,2,\cdots,r\}.$ When
$a_{i_j}a_{i_{j+1}}<0$, we say that $a_{i_j}$ and $a_{i_{j+1}}$ have
a variation of sign. If the number of variations of signs is $m$,
then $p(x)$ has at most $m$ positive real roots. Moreover, it is
always possible to choose the coefficients of $p(x)$ in such a way
that $p(x)$ has exactly $r-1$ positive real roots.
\end{theorem}

For a proof of Descartes Theorem see pages 82--83 of \cite{Berezin}.

\medskip

So from Descartes Theorem we can choose $v_6,v_4,v_2,v_0$ in order
that the $G_{20}$ has 3,2,1 or 0 real positive roots. This completes
the proof of the first part of  Theorem \ref{main}.

\begin{remark}\label{r2}
Again the exact definition for the numerical coefficients which
appear in $v_6,v_4,v_2$ and $v_0$ are given in Appendices B and C.
For instance
$$
v_0=\displaystyle{\int\limits_{0}\limits^{2\pi}}A_5
d\theta+\displaystyle{\int\limits_{0}\limits^{2\pi}}s_{5,3}
d\theta=\displaystyle{\int\limits_{0}\limits^{2\pi}}A_5 d\theta
=665.2264933...
$$
\end{remark}

For completing the proof of Theorem \ref{main} we shall provide
examples of system \eqref{perturbed} with 3, 2, 1 and 0 limit
cycles. In fact, strictly speaking it is not necessary to provide
examples with 3, 2, 1 and 0 limit cycles but we want to provide such
examples.

\section{Examples}\label{sexamples}

\noindent{\bf Example with 3 limit cycles}\\
In Figure \ref{tres} we see that  for $\varepsilon=0.001$ the system
\begin{equation}
\begin{array}{ll}\label{tris}
\dot{x}=&-y(3x^2+y^2)+\varepsilon+\varepsilon^2(3570.576292
x-752.8823806 x^3)\\
&=y(3x^2+y^2)+0.001+0.003570576292x-0.0007528823806x^3,\\
\\
\dot{y}=&x(x^2-y^2)+\varepsilon(1-37.74385845y^2)\\
&=x(x^2-y^2)+0.001-0.03774385845 y^2,
\end{array}
\end{equation}
has three limit cycles, since for system \eqref{tris} we have
$$
G_{20}(r_0)=- 1182.624878\,{ r_0}+ 4139.187071\,\dfrac{1}{r_0}-
3621.788686\,\dfrac{1}{r_0^3} + 665.2264933\, \dfrac{1}{r_0^5},
$$
and from $G_{20}(r_0)=0$ we obtain the three positive roots near to
$r_0=0.5, 1, 1.5.$

\begin{figure}[ht!]
$$
\begin{array}{cc}
{\includegraphics[width=5cm]{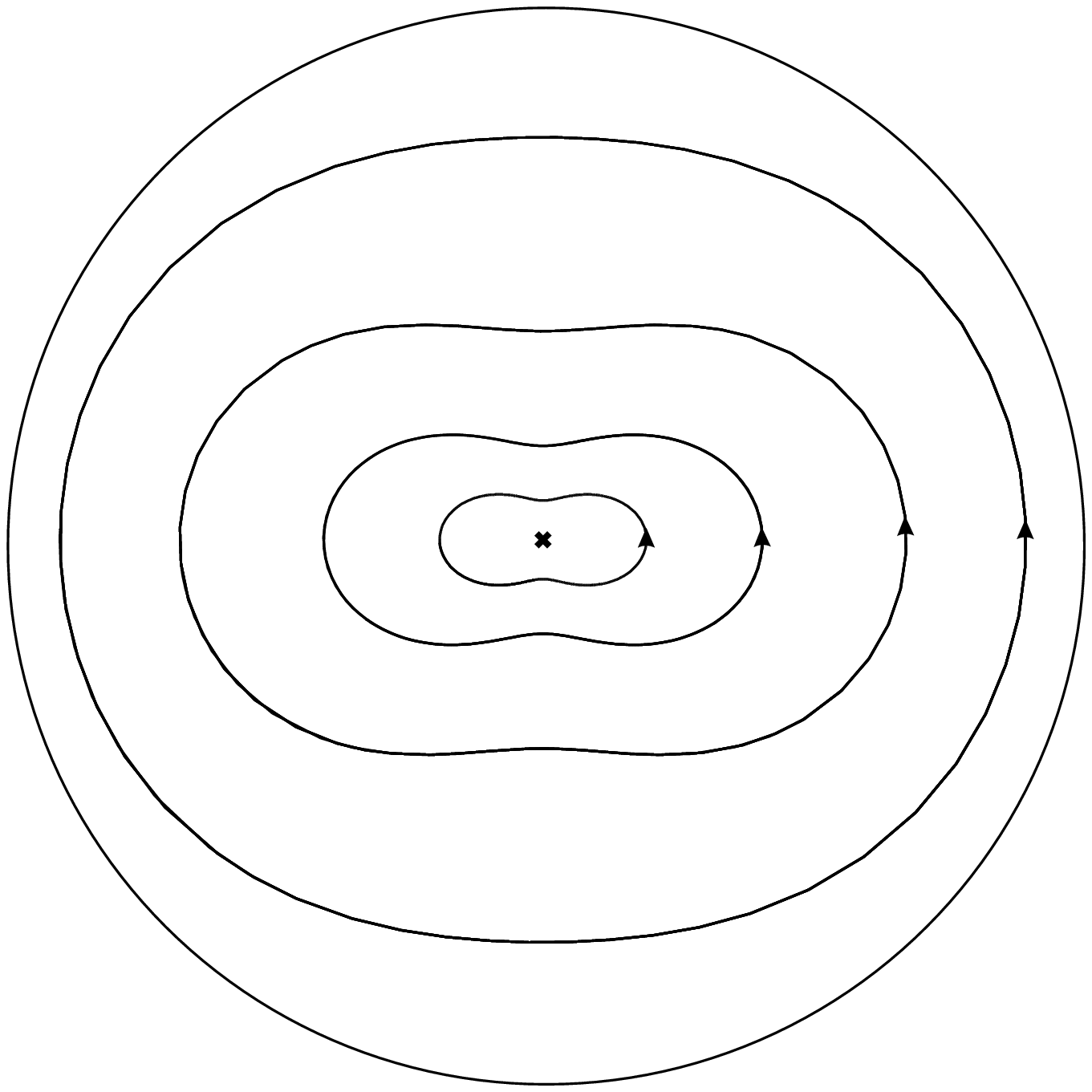}}&\qquad
{\includegraphics[width=5cm]{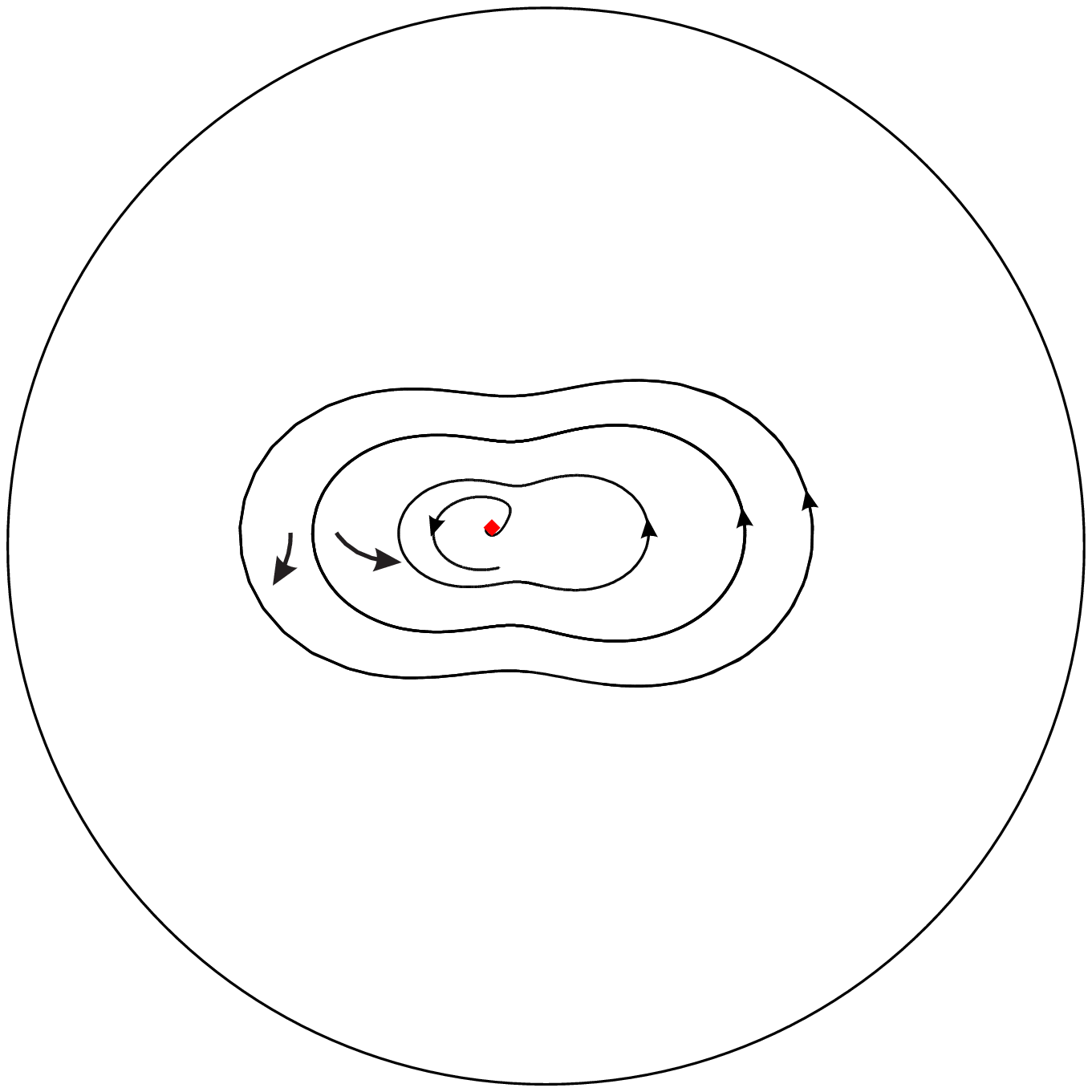}}\\
\varepsilon=0 & \qquad \varepsilon=0.001
\end{array}
$$
\caption{For $\varepsilon=0$ we have the degenerate center of system
\eqref{1}, and for $\varepsilon=0.001$ the perturbed system
\eqref{tris} has three limit cycles.} \label{tres}
\end{figure}
We have used the program P4 described in Chapters 9 and 10 of
\cite{DLLA} for doing the phase portraits in the Poincar\'e disc
which appear in this paper.

\medskip
\noindent{\bf Example with 2 limit cycles}\\
For $\varepsilon=0.001$  the system
\begin{equation}\label{dio}
\begin{array}{ll}
\dot{x}=&-y(3x^2+y^2)+\varepsilon(1+y+x^2y)+\varepsilon^2(-856.6373973x+y^3)\\
          &=-y(3x^2+y^2)+0.001+0.001y+0.001x^2y-0.0008566373973x+0.000001y^3,\\
          \\
  \dot{y}=&x(x^2-y^2)+\varepsilon(1+y^2)+\varepsilon^2(x+73.80732101y^3)\\
  &=x(x^2-y^2)+0.001+0.001y^2+0.000001x+0.00007380732101y^3,
\end{array}
\end{equation}
gives
$$
G_{20}(r_0)= 231.8725375\,{ r_0}- 993.0560642\, \dfrac{1}{r_0}+
95.95703341\, \dfrac{1}{r_0^3}+ 665.2264933\, \dfrac{1}{r_0^5},
$$
and $G_{20}(r_0)=0$ has the two positive zeros  $r_0=1$ and $r_0=2.$
In Figure \ref{fdio} we see the two limit cycles bifurcated from the
degenerate center of the unperturbed system \eqref{dio}.
\begin{figure}[ht!]
$$
\begin{array}{c}
{\includegraphics[width=5cm]{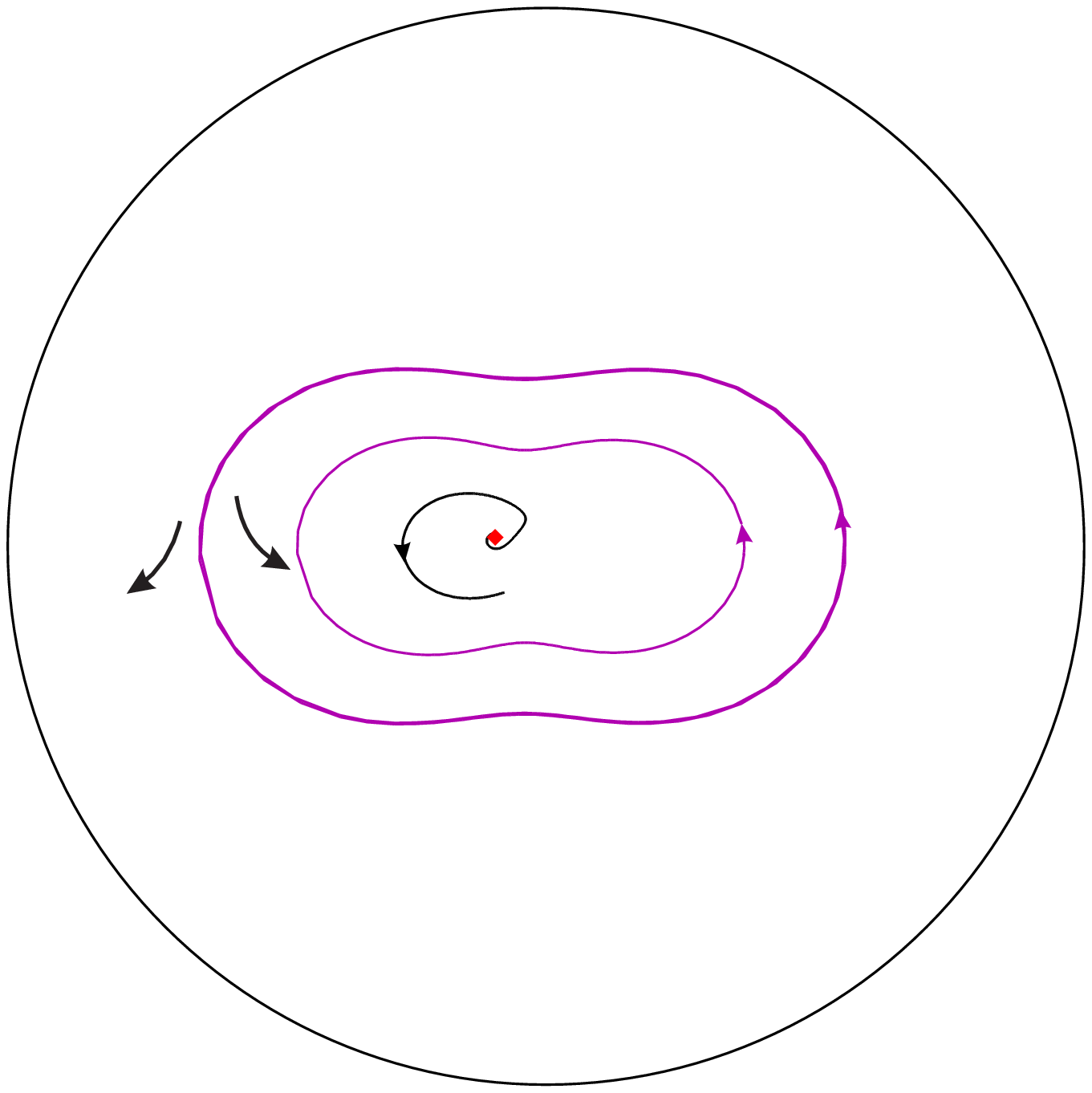}}\\
 \qquad \varepsilon=0.001
\end{array}
$$
\caption{Two limit cycles bifurcate from the degenerate center of
the unperturbed system \eqref{dio}.} \label{fdio}
\end{figure}

\medskip
\noindent{\bf Example with 1 limit cycle}\\
For $\varepsilon=0.001$ the system
\begin{equation}\label{ena}
\begin{array}{ll}
\dot{x}&=-y(3x^2+y^2)+\varepsilon(1-176.5322447x)+\varepsilon^2(x+x^3)\\
 &=-y \left( 3\,{x}^{2}+{y}^{2} \right) + 0.001- 0.1765312447\,x+
 0.000001\,{x}^{3},\\
 \\
 \dot{y}&=x(x^2-y^2)+\varepsilon(10+10y+5xy)+\varepsilon^2(y-y^3)\\
&=x \left( {x}^{2}- {y}^{2} \right) + 0.010+ 0.010001\,y+ 0.005\,xy
- 0.000001\,{y}^{3},
\end{array}
\end{equation}
has
$$
G_{20}(r_0)=-1.570796327r_0+21.62373221\dfrac{1}{r_0}-5545.821670
\dfrac{1}{r_0^3}+6652.264933\dfrac{1}{r_0^5}.
$$
From relation $G_{20}(r_0)=0$ we obtain  $-1.097575824,$ $
1.097575824, $ $ -5.725902515-5.148324797i,$ $
5.725902515+5.148324797i,$$ -5.725902515+5.148324797i, $$
5.725902515-5.148324797i.$  So only one limit cycle can  bifurcate
from a periodic orbit of the center of the unperturbed system
\eqref{ena}, as we can see in Figure \ref{fena}.
\begin{figure}[ht!]
$$
\begin{array}{c}
{\includegraphics[width=5cm]{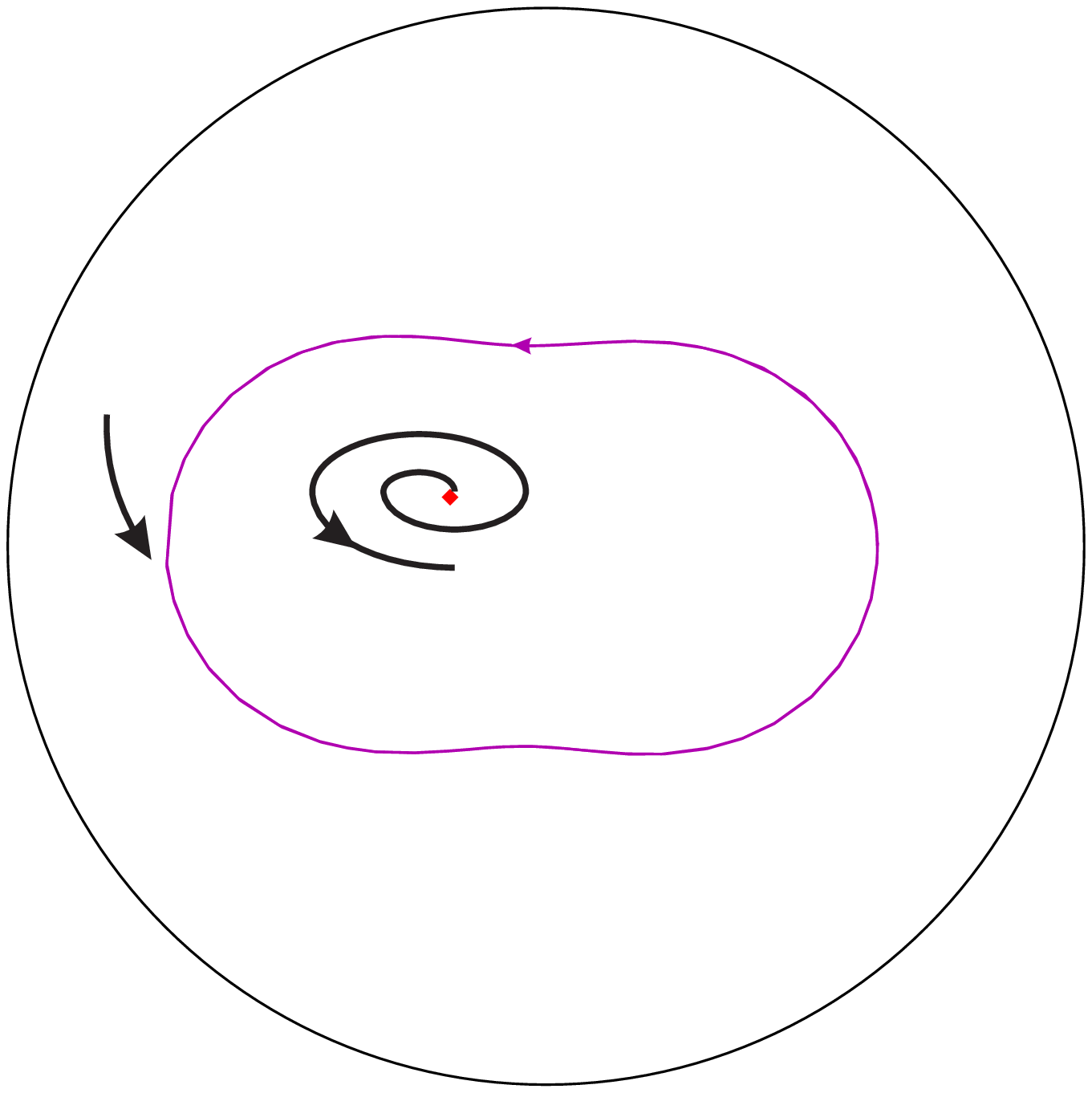}}\\
 \qquad \varepsilon=0.001
\end{array}
$$
\caption{The limit cycle bifurcated from the degenerate center of
the unperturbed system \eqref{ena}.} \label{fena}
\end{figure}

\medskip

\noindent{\bf Example with zero limit cycles}\\
Now for $\varepsilon=0.001$ we consider system
\begin{equation}\label{mhden}
\begin{array}{ll}
\dot{x}&=-y(3x^2+y^2)+\varepsilon+\varepsilon^2x\\
 &=-y \left( 3\,{x}^{2}+{y}^{2} \right) + 0.001+ 0.000001\,x,\\
 \\
 \dot{y}&=x(x^2-y^2)+\varepsilon(1+y^2)+\varepsilon^2y^3\\
 &=x \left( {x}^{2}- {y}^{2} \right) + 0.001+ 0.001\,{y}^{2}+
 0.000001\,{y}^{3},
\end{array}
\end{equation}
with
$$
G_{20}(r_0)=3.141592654r_0+1.159249021\dfrac{1}{r_0}+95.95703341
\dfrac{1}{r_0^3}+665.2264933\dfrac{1}{r_0^5}.
$$
We have that $G_{20}(r0)=0$ has solutions
$-2.116012294-1.570359831i, 2.116012294+1.570359831i, -2.095699520i,
2.095699520i, -2.116012294+1.570359831i, 2.116012294-1.570359831i.$
So no limit cycles can bifurcate from the degenerate center, see
also Figure \ref{fmhden}.
\begin{figure}[ht!]
$$
\begin{array}{c}
{\includegraphics[width=4cm,height=4cm]{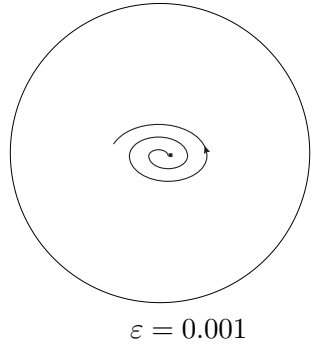}}\\
 \qquad \varepsilon=0.001
\end{array}
$$
\caption{No limit cycle bifurcates from the degenerate center of the unperturbed system \eqref{mhden}.} \label{fmhden}
\end{figure}

%%%%%%%%%%%%%%%%%%%%%%%%%%%%%%%%%%%%%%%%%%%%%%%%%%%%%%%%%%%%%%%%%%%%%%%%%%%%%%%%%%%%%%%%%%%%%%%%%%%%%%%%
%
%\newpage
\section{Appendix A}\label{Appendix A}
$$
\small{
\begin{array}{ll}
g_{1,1}(\theta)=&\cos \theta\left( 2\, \cos^5 \theta
-1 \right)a_{00}+\sin  \theta   \left( 1+2\,  \cos^2\theta \right) ^{2}c_{00}, \\
 &\\

g_{1,2}(\theta)=&\sin^2 \theta \left( 1+2\, \cos^2 \theta
\right)c_{01}+\sin \theta\cos \theta
 \left( 1+2\, \cos^2 \theta \right)c_{{10}}\\
 &+
 \sin \theta\cos \theta
 \left( 2\, \cos^2 \theta-1 \right) a_{{01}}+
 \cos^2 \theta \left( 2\,
 \cos^2 \theta-1 \right) a_{{10}},\\

&\\
 g_{1,3}(\theta)=& \sin^3 \theta \left( 1+2\,
 \cos^2 \theta \right)c_{{02}}+\left( -2\, \cos^5 \theta+ \cos^3 \theta+\cos \theta
 \right) c_{{11}}\\
 &+ \cos^2 \theta\sin \theta\left( 1+2\, \cos^2 \theta \right)c_{{20}}+
 \left( -2\, \cos^5 \theta+3\,
 \cos^3 \theta-\cos \theta\right) a_{{02}}+\\
 &\cos^2 \theta\sin \theta\left( 2\, \cos^2 \theta-1 \right) a_{{11}}+ \cos^3 \theta \left( 2\,
 \cos^2 \theta-1 \right)a_{{20}},\\
\end{array}
}
$$
$$
\small{
\begin{array}{l}
+ \left( -4\,
 \cos^7 \theta-2\, \cos^3 \theta+6\, \cos^5 \theta \right) a_{{21}}a_{{00}}
 + \left( -2\, \cos^5 \theta-2\, \cos^3 \theta
+4\, \cos^7 \theta \right) c_{{21}}c_{{00}}\\
+ \left( -2\,\cos \theta+6\, \cos^5 \theta-4\,
 \cos^7 \theta \right) c_{{03}}c_{{0
0}} + \left( 4\, \cos^7 \theta-6\,
 \cos^5 \theta+ \cos^3 \theta+\cos \theta\right)
c_{{1
1}}a_{{01}}\\
+ \left( 4\,  \cos^7 \theta-6\, \cos^5 \theta+ \cos^3 \theta+\cos
\theta
 \right) c_{{10}}a_{{02}}
 + \left( 4\, \cos^7 \theta-6\, \cos^5 \theta+ \cos^3 \theta+\cos \theta\right) c_{{02}}a_{{10}}\\
+ \left( 4\, \cos^7 \theta-6\, \cos^5 \theta+ \cos^3 \theta +\cos
\theta\right) c_{{01}}a_{{11}} + \left( -2\, \cos \theta+6\, \cos^5
\theta-4\, \cos^7 \theta\right)
  c_{{02}}c_{{01}}\\
 + \left( -4\, \cos^7 \theta-2\, \cos^3 \theta+6\, \cos^5 \theta \right) a_{{2 0}}a_{{01}} + \left( -4\,
\cos^7 \theta-2\, \cos^3 \theta+6\, \cos^5 \theta \right) a_{{11}}a_{{10}}\\
+
 \left( -2\, \cos^5 \theta-2\,
 \cos^3 \theta+4\, \cos^7 \theta \right) c_{{20}}c_{{01}}+
 \left( -2\, \cos^5 \theta-2\,
 \cos^3 \theta+4\, \cos^7 \theta \right) c_{{11}}c_{{10}},
\end{array}
}
$$

$$
\small{
\begin{array}{ll}
g_{2,5}(\theta)=&-c_{{10}}a_{{30}} \cos^4 \theta
 \left( -2\, \cos^2 \theta-1+4\,
 \cos^4 \theta \right)
  -c_{{30}}a_{
{10}} \cos^4 \theta \left( -2\,
 \cos^2 \theta-1+4\, \cos^4 \theta \right),
\\
 & -{c_{{02}}}^{2} \sin^5 \theta\cos \theta
 \left( 2\, \cos^2 \theta+1 \right)
-{c_{{11}}}^{2} \sin^3 \theta
 \cos^3 \theta \left( 2\, \cos^2 \theta+1 \right)
\\
&-{c_{{20}}}^{2}
 \cos^5 \theta\sin \theta\left( 2\, \cos^2 \theta+1
 \right)
 +{a_{{02}}}^{2}  \sin^5 \theta \cos \theta\left( 2\, \cos^2 \theta-1 \right)
\\
 & + \left( -7\, \cos^4 \theta+1-4\, \cos^8 \theta+10\, \cos^6 \theta
 \right) a_{{01}}c_{{03}}
 + \left( -7\, \cos^4 \theta+1-4\, \cos^8 \theta+10\, \cos^6 \theta
 \right) a_{{03}}c_{{01}}
\\
 &+ \left( -7\, \cos^4 \theta+1-4\, \cos^8 \theta+10\, \cos^6 \theta
 \right) a_{{02}}c_{{02}}
 +d_{{10}}\sin \theta\cos \theta\left( 2\, \cos^2 \theta+1 \right)
\\
 &+b_{{10}} \cos^2 \theta \left( 2\, \cos^2 \theta-1 \right) +d_{{01}}  \sin^2 \theta
^{2 } \left( 2\, \cos^2 \theta+1
 \right)
 +{a_{{11}}}^{2} \sin^3 \theta \cos^3 \theta \left( 2\,
\cos^2 \theta-1 \right)
\\
 &-c_{{01}}a_{{12}}
 \sin^3 \theta\cos \theta\left( -2\, \cos^2 \theta-
1+4\, \cos^4 \theta \right)
 -c_{{30
}}a_{{01}}\sin \theta\cos^3 \theta \left( -2\, \cos^2 \theta-1+4\,
\cos^4 \theta
 \right)\\
 &-c_{{21}}a_{{10}}\sin \theta\cos^3 \theta \left( -2\, \cos^2 \theta-1+4\, \cos^4 \theta \right)
 -c_{{20}}a_{{11}}\sin \theta
 \cos^3 \theta \left( -2\, \cos^2 \theta-1+4\, \cos^4 \theta \right)
\\
\end{array}
}
$$

$$
\small{
\begin{array}{ll}

 &+b_{{20}} \cos^3 \theta \left( 2\, \cos^2 \theta-1 \right)
 -2\,c_{{12}}c_{{02}} \sin^5 \theta \cos^2 \theta \left( 2\, \cos^2 \theta+1 \right)\\

  &-c_{{12}}a_{{11}} \sin^3 \theta \cos^2 \theta
 \left( -2\, \cos^2 \theta-1+4\,
 \cos^4 \theta \right)
  -c_{{30}}a_{
{11}} \cos^4 \theta\sin \theta\left( -2\, \cos^2 \theta-1+4\, \cos^4
\theta
 \right)\\
 &-c_{{21}}a_{{20}} \cos^4 \theta\sin \theta\left( -2\, \cos^2 \theta-1+4\, \cos^4 \theta \right)
 +2\,a_{{30}}a_{{20}} \cos^6 \theta\sin \theta
 \left( 2\, \cos^2 \theta-1 \right)\\
 %\end{array}
%}
%$$
%$$
%\small{
%\begin{array}{ll}
&-c_{{03}}a_{{20}} \sin^3 \theta
 \cos^2 \theta \left( -2\, \cos^2 \theta-1+4\, \cos^4 \theta \right)
 -c_{{02}}a_{{21}} \sin^3 \theta \cos^2 \theta \left( -2\, \cos^2 \theta-1+4\, \cos^4 \theta
 \right)\\
 &+2\,a_{{30}}a_{{02}} \sin^3 \theta \cos^4 \theta \left( 2
\, \cos^2 \theta-1 \right) -c_{{03} }a_{{02}} \sin^5 \theta \left(
-2\,
 \cos^2 \theta-1+4\, \cos^4 \theta \right)\\
  &+d_{{20}} \cos^2 \theta\sin \theta
 \left( 2\, \cos^2 \theta+1 \right)
+b_{{11}} \cos^2 \theta\sin \theta\left( 2\, \cos^2 \theta
-1\right) \\
&-c_{{02}}a_{{03}} \sin^5 \theta \left( -2\, \cos^2 \theta-1+4\,
\cos^4 \theta \right) +
 \left( \cos \theta+10\, \cos^7 \theta-7\, \cos^5 \theta-4\, \cos^9 \theta \right)
c_{{1
1}}a_{{03}}\\
&+ \left( \cos \theta+10\, \cos^7 \theta-7\, \cos^5 \theta-4\,
\cos^9 \theta \right) c_{{03}}a_{{11}} +
\left(  \cos^5 \theta-6\, \cos^7 \theta+4\, \cos^9 \theta+ \cos^3 \theta \right) c_{{21}}a_{{11}}\\

 &+
 \left(  \cos^5 \theta-6\, \cos^7 \theta+4\, \cos^9 \theta+  \cos^3 \theta\right)  c_{{11}}a_{{21}} + \left(  \cos^5 \theta-6\, \cos^7 \theta+4\, \cos^9 \theta+ \cos^3 \theta \right) c_{{20}}a_{{12}}\\
 &+
 \left(  \cos^5 \theta-6\, \cos^7 \theta+4\, \cos^9 \theta  +  \cos^3 \theta\right) c_{{12}}a_{{20}} + \left(  \cos^5 \theta-6\, \cos^7 \theta+4\, \cos^9 \theta+ \cos^3 \theta \right) c_{{02}}a_{{30}}
 \\
 &+
 \left( \cos \theta+10\, \cos^7 \theta-7\, \cos^5 \theta-4\, \cos^9 \theta \right)
c_{{1 2}}a_{{02}} + \left( \cos \theta+10\, \cos^7 \theta -7\,
\cos^5 \theta-4\, \cos^9 \theta \right)c_{{02}}a_{{12}}
\\
\end{array}
}
$$

$$
\small{
\begin{array}{ll}
g_{2,7}(\theta)=&-{c_{{21}}}^{2} \sin^3 \theta
 \cos^5 \theta \left( 2\, \cos^2 \theta+1 \right) +{a_{{12}}}^{2}
 \sin^5 \theta \cos^3 \theta \left( -1+2\, \cos^2 \theta \right)\\
  &-{c_{{03}}}^{2}  \sin^7 \theta \cos \theta\left( 2\,
 \cos^2 \theta+1 \right)
  +{a_{{03}}
}^{2} \sin^7 \theta \cos \theta\left( -1+2\, \cos^2 \theta \right)\\
  &+{a_{{21}}}^{2} \sin^3 \theta \cos^5 \theta
 \left( -1+2\, \cos^2 \theta
 \right)
  -{c_{{12}}}^{2} \sin^5 \theta \cos^3 \theta \left( 2\,
\cos^2 \theta+1 \right)
-{c_{{30}}}^{2}\\
 &\cos^7 \theta\sin \theta\left( 2\, \cos^2 \theta+1
 \right)
 +2\,a_{{30}}a_{{12}} \sin^3 \theta \cos^5 \theta \left( -
1+2\, \cos^2 \theta \right)\\
 &-c_{{03
}}a_{{12}} \sin^5 \theta\cos \theta\left( 4\, \cos^4 \theta-2\,
\cos^2 \theta-1
 \right)
 -c_{{30}}a_{{21}} \cos^5 \theta\sin \theta\left( 4\, \cos^4 \theta-2\, \cos^2 \theta-1 \right)\\
 &-c_{{12}}a_{{21}} \sin^3 \theta \cos^3 \theta \left( 4\, \cos^4 \theta-2\, \cos^2 \theta-1
 \right)
 -c_{{30}}a_{{03}} \sin^3 \theta \cos^3 \theta \left( 4
\, \cos^4 \theta-2\, \cos^2 \theta-1 \right)\\
 &-2\,c_{{21}}c_{{03}}
 \sin^5 \theta \cos^3 \theta \left( 2\, \cos^2 \theta+1 \right)
  -2\,c_{{12}}c_{{03}} \sin^6 \theta \cos^2 \theta \left( 2\, \cos^2 \theta+1 \right)\\
 &+2\,a_{{30}}a_{{03}} \sin^4 \theta \cos^4 \theta \left( -1+2\, \cos^2 \theta \right)
  +2\,a_{{12}}a_{{03}} \left( \sin \theta\right) ^{6} \cos^2 \theta
 \left( -1+2\, \cos^2 \theta
 \right)\\
\end{array}
}
$$
%%%%%%%%%%%%%%%%%%%%%%%%%%%%%%%%%%%%%%%%%%%%%%%%%%%%%%%%%%%%%%%%%%%%%%%%%%%%%%%%%%%%%

\section{Appendix B}
Here we present the expressions of
$A_4,A_2,A_0,A_5,A_3,A_1,\tilde{A}_1$ that appear in equation
\eqref{ints}.
$$
\small{
\begin{array}{ll}
A_4(\theta)=&{{\rm e}^{5\,  \sin^2 \theta }}\left[\left[ -{\dfrac {
\left( 4\,a_{01}c_{00}{I_2}-8\,{I_1}\,a_
{{01}}c_{00}-8\,{I_1}\,c_{00}c_{10}-4\,a_{00}c_{{01}
}{I_3}+4\,c_{00}c_{10}{I_2} \right)   \cos^4 \theta }{{I_2}-2\,{I_1}}}\right.\right.\\
&-{\dfrac { \left(
-4\,{I_1}\,c_{00}c_{10}+2\,a_{00}c_{01}{I_3}-2\,a_{01}c_
{{00}}{I_2}+2\,c_{00}c_{10}{I_2}+4\,{I_1}\,a_{01}c_{00} \right)
\cos^2 \theta }{{
I_2}-2\,{I_1}}}\\
&-\left.{\dfrac { \left(
-a_{01}c_{00}{I_2}+2\,{I_1}\,a_{01}c_{00}+2\,{I_1}\,a_{00}c_{01}-a_{00}c_{
{01}}{I_2} \right) }{{I_2}-2\,{I_1}}} \right] \sin
 \theta\\

 & -{\dfrac { \left( 4\,c_{00}c_{01}{I_3}+4\,a_{{00}}c_{10}{I_2}-8\,{I_1}\,a_{00}a_{01}-8\,{I_1}\,
a_{00}c_{10}+4\,a_{00}a_{01}{I_2} \right)  \left( \cos
 \theta  \right) ^{5}}{{I_2}-2\,{I_1}}}\\
 &-{\dfrac{  \left( 12
\,{I_1}\,a_{00}a_{01}+4\,{I_1}\,a_{00}c_{10}-2\,c_
{{00}}c_{01}{I_3}-2\,a_{00}c_{10}{I_2}-6\,a_{00}a _{{01}}{I_2}
\right)  \left( \cos \theta  \right) ^{
3}}{{I_2}-2\,{I_1}}}\\
&-\left.{\dfrac { \left( -2\,c_{00}c_{01}{I_1}+c_
{{00}}c_{01}{I_2}-4\,{I_1}\,a_{00}a_{01}+2\,a_{{00}
}a_{01}{I_2}+2\,{I_1}\,a_{00}c_{10}-c_{00}c_{{01}
}{I_3}-a_{00}c_{10}{I_2} \right) \cos \theta
}{{I_2}-2\,{I_1}}}\right],\\
\\
%\end{array}
%}
%$$
%
%\bigskip
%$$
%\small{
%\begin{array}{ll}
A_2(\theta)=&\left[A_{2,7}\cos^7\theta +A_{2,6}\cos^6\theta
\sin\theta +A_{2,5}\cos^5\theta +A_{2,4}\cos^4\theta \sin\theta +
A_{2,3}\cos^3\theta +A_{2,2}\cos^2\theta \sin\theta  \right.\\
&+\left.A_{2,1}\cos\theta +A_{2,0}\theta \sin\theta
\right]\exp(3\sin^2\theta ),\\
\\
%\end{array}
%}
%$$
%\bigskip
%$$
%\small{
%\begin{array}{ll}
A_{2,7}=&4\,a_{03}a_{00}-4\,a_{21}a_{00}-4\,c_{30}a_{00}+4\,c
_{{12}}a_{00}-4\,a_{20}a_{01}+4\,a_{02}a_{01}+4\,c_{{1
1}}a_{01}-4\,c_{10}a_{20}\\
&+4\,c_{10}a_{02}+4\,{\dfrac {a
_{{11}}c_{01}{I_3}}{2\,{I_1}-{I_2}}}+4\,c_{00}a_{{
12}}+4\,c_{03}c_{00}+8\,c_{21}c_{00}+4\,{\dfrac {c_{01}
c_{20}{I_3}}{2\,{I_1}-{I_2}}}\\
&-4\,{\dfrac {c_{01}c_{{0
2}}{I_3}}{2\,{I_1}-{I_2}}}+4\,c_{11}c_{10},\\
\\
A_{2,6}=&-4\,a_{12}a_{00}-4\,c_{03}a_{00}-8\,c_{21}a_{00}-4\,
a_{11}a_{01}-4\,c_{20}a_{01}+4\,c_{02}a_{01}-4\,{
\dfrac {a_{20}c_{01}{I_3}}{2\,{I_1}-{I_2}}}\\
&+4\,{ \dfrac {a_{02}c_{01}{I_3}}{2\,{I_1}-{I_2}}}-4\,a_{{1
1}}c_{10}+4\,a_{03}c_{00}-4\,a_{21}c_{00}-4\,c_{30}
c_{00}+4\,c_{12}c_{00}\\
&+4\,{\dfrac {c_{01}c_{11}{I_3}
}{2\,{I_1}-{I_2}}}-4\,c_{20}c_{10}+4\,c_{10}c_{02},\\
\\
%\end{array}
%}
%$$
%$$
%\small{
%\begin{array}{ll}
A_{2,5}=&-10\,a_{03}a_{00}+6\,a_{21}a_{00}+2\,c_{30}a_{00}-6
\,c_{12}a_{00}+6\,a_{20}a_{01}-10\,a_{02}a_{01}-6\,c
_{{11}}a_{01}+2\,c_{10}a_{20}\\
&-6\,c_{10}a_{02}-6\,{ \dfrac
{a_{11}c_{01}{I_3}}{2\,{I_1}-{I_2}}}-6\,c_{{0
0}}a_{12}+2\,c_{03}c_{00}-4\,c_{21}c_{00}-2\,{\dfrac {c
_{{01}}c_{20}{I_3}}{2\,{I_1}-{I_2}}}\\
&+6\,{\dfrac {c_{{0
1}}c_{02}{I_3}}{2\,{I_1}-{I_2}}}-2\,c_{11}c_{{10} },
\\  \\
A_{2,4}=&6\,a_{00}a_{12}-2\,a_{00}c_{03}+4\,a_{00}c_{21}+6\,a
_{{01}}a_{11}+2\,a_{01}c_{20}-6\,a_{01}c_{02}+2\,{
\dfrac {a_{20}c_{01}{I_3}}{2\,{I_1}-{I_2}}}\\
&-6\,{ \dfrac {a_{02}c_{01}{I_3}}{2\,{I_1}-{I_2}}}+2\,a_{{1
1}}c_{10}-6\,a_{03}c_{00}+2\,a_{21}c_{00}-2\,c_{00}
c_{30}-2\,c_{00}c_{12}-2\,{\dfrac {c_{01}c_{11}{I_3}
}{2\,{I_1}-{I_2}}}\\
&-2\,c_{10}c_{20}-2\,c_{02}c_{{10} },
\\  \\
A_{2,3}=&2\,b_{{00}}+8\,a_{00}a_{03}-2\,a_{00}a_{21}+a_{00}c_{{
30}}+a_{00}c_{12}-2\,a_{01}a_{20}+8\,a_{01}a_{02}+a
_{{01}}c_{11}+a_{20}c_{10}\\
&+a_{02}c_{10}-{\dfrac {
 \left( -2\,{I_3}-{I_2}+2\,{I_1} \right) c_{01}a_{{11}
}}{2\,{I_1}-{I_2}}}+a_{12}c_{00}-2\,c_{00}c_{03}-
3\,c_{00}c_{21}\\
&-{\dfrac { \left( -{I_2}+2\,{I_1}+{I_3 } \right)
c_{01}c_{20}}{2\,{I_1}-{I_2}}}+{\dfrac {
 \left( 2\,{I_1}-{I_3}-{I_2} \right) c_{02}c_{01}}{-
{I_2}+2\,{I_1}}}-2\,c_{10}c_{11},\\
\\
A_{2,2}=&2\,d_{{00}}-2\,a_{00}a_{12}+a_{00}c_{03}+a_{00}c_{{21
}}-2\,a_{01}a_{11}+a_{01}c_{20}+a_{01}c_{02}+a_{{20
}}c_{01}\\
&- {\dfrac { \left( -2\,{I_3}-{I_2}+2\,{I_1}
 \right) c_{01}a_{02}}{2\,{I_1}-{I_2}}}+a_{11}c_{{10
}}+a_{03}c_{00}+a_{21}c_{00}-2\,c_{00}c_{12}\\
&-{\dfrac { \left( -{I_2}+2\,{I_1}+{I_3} \right) c_{11}c_{01}}
{2\,{I_1}-{I_2}}}-2\,c_{02}c_{10},\\
\\

A_{2,1}=&-b_{{00}}-2\,a_{00}a_{03}+a_{00}c_{12}-2\,a_{01}a_{{0
2}}+a_{01}c_{11}+a_{02}c_{10}+a_{11}c_{01}+a_{12}
c_{00}-2\,c_{00}c_{03}\\
&-{\dfrac { \left( -{I_2}+2\,{I_1
}+{I_3} \right) c_{01}c_{02}}{2\,{I_1}-{I_2}}},\\
\\
A_{2,0}=&a_{03}c_{00}+d_{{00}}+a_{00}c_{03}+a_{01}c_{02}+a_{
{02}}c_{01}.
\end{array}
}
$$
Now we have
$$
\small{
\begin{array}{l}
A_0(\theta)=A_{0,9}\cos^9\theta +A_{0,8}\cos^8\theta +A_{0,7}\cos^7\theta +A_{0,6}\cos^6\theta +A_{0,5}\cos^5\theta +A_{0,4}\cos^4\theta \\
+A_{0,3}\cos^3\theta +A_{0,2}\cos^2\theta +A_{0,1}\cos\theta
+A_{0,0},
\end{array}
}
$$
with
$$
\small{
\begin{array}{ll}
A_{0,9}(\theta)=&- \left(
-4\,a_{21}c_{11}+4\,c_{02}c_{03}+8\,c_{02}c_{{
21}}-4\,c_{03}c_{20}+4\,c_{11}c_{12}-4\,c_{11}c_{{30}
}-8\,c_{20}c_{21}\right. \\

&-8\,a_{11}c_{21}+4\,a_{12}c_{02}-4
\,a_{12}c_{20}+4\,a_{20}a_{21}-4\,a_{20}c_{12}+4\,a_
{{20}}c_{30}+4\,a_{02}a_{03}-4\,a_{02}a_{21}\\

&\left.+4\,a_{{0
2}}c_{12}-4\,a_{02}c_{30}-4\,a_{03}a_{20}+4\,a_{03}c
_{{11}}-4\,a_{11}a_{12}-4\,a_{11}c_{03} \right) {{\rm e}^
{  \sin^2 \theta }},\\
\\
A_{0,8}(\theta)=&- \left(
-4\,a_{03}c_{20}+4\,a_{11}a_{21}-4\,a_{11}c_{{
12}}+4\,a_{12}a_{20}-4\,a_{02}c_{03}-8\,a_{02}c_{{21}
}-4\,a_{03}a_{11}\right.\\
&+4\,a_{03}c_{02}-8\,c_{11}c_{21}-4
\,c_{12}c_{20}+4\,c_{20}c_{30}+4\,a_{11}c_{30}-4\,a_{02}a_{12}-4\,a_{12}c_{11}\\
&\left. +4\,a_{20}c_{03}+8\,a_{{20
}}c_{21}-4\,a_{21}c_{02}+4\,a_{21}c_{20}+4\,c_{02}c_
{{12}}-4\,c_{02}c_{30}-4\,c_{03}c_{11} \right) {{\rm e}^{
  \sin^2 \theta }}\sin \theta,\\
  \\
 %\end{array}
% }
%$$
%
%$$
%\small{
%\begin{array}{ll}
 A_{0,7}(\theta)=&- \left( 2\,c_{11}c_{30}+4\,c_{20}c_{21}-14\,a_{02}a_{{0
3}}+10\,a_{02}a_{21} -10\,a_{02}c_{12}+6\,a_{02}c_{{30
}}+10\,a_{03}a_{20}\right.\\

&-10\,a_{03}c_{11}+10\,a_{11}a_{{12}
}+2\,a_{11}c_{03}+12\,a_{11}c_{21}-10\,a_{12}c_{02}+
6\,a_{12}c_{20}-6\,a_{20}a_{21}+6\,a_{20}c_{12}\\
&\left.-2\,a _{{20}}c_{30}+6\,a_{21}c_{11}-2\,c_{02}c_{03}-12\,c_{{0
2}}c_{21}-2\,c_{03}c_{20}-6\,c_{11}c_{12} \right) {
{\rm e}^{  \sin^2 \theta }},\\
\\
A_{0,6}(\theta)=&- \left(
2\,c_{20}c_{30}+2\,a_{02}c_{03}+6\,a_{12}c_{{1
1}}+2\,a_{20}c_{03}-4\,a_{20}c_{21}+6\,a_{21}c_{02}
-2\,a_{21}c_{20}-6\,c_{02}c_{12}\right.\\
&+2\,c_{02}c_{30}-2\,
c_{03}c_{11}+4\,c_{11}c_{21}+2\,c_{12}c_{20}-6\,a_{{
11}}a_{21}+6\,a_{11}c_{12}-2\,a_{11}c_{30}-6\,a_{{12}
}a_{20}\\
&\left.+6\,a_{03}c_{20}+12\,a_{02}c_{21}+10\,a_{03}a
_{{11}}+10\,a_{02}a_{12}-10\,a_{03}c_{02} \right) {
{\rm e}^{  \sin^2 \theta }}\sin \theta,\\
\\
A_{0,5}(\theta)=&- \left[
-2\,b_{20}+18\,a_{02}a_{03}-8\,a_{11}a_{12}+2
\,b_{02}-a_{20}c_{12}-a_{20}c_{30}-a_{21}c_{11}-4
\,c_{02}c_{03}+c_{02}c_{21}\right.\\
&+2\,c_{03}c_{20}+2\,c_{{1 ,1}}c_{30}
+3\,c_{20}c_{21}+3\,a_{11}c_{03}-3\,a_{11}
c_{21}+7\,a_{12}c_{02}-a_{12}c_{20}+2\,a_{20}a_{{21
}}\\
&\left.+7\,a_{02}c_{12}-a_{02}c_{30}+2\,d_{11}-8\,a_{02}a
_{{21}}-8\,a_{03}a_{20}+7\,a_{03}c_{11} \right]
{\rm e}^{  \sin^2 \theta },\\
%\end{array}
%}
%$$
%
%$$
%\small{
%\begin{array}{ll}
A_{0,4}(\theta)=&- \left[
7\,a_{03}c_{02}-a_{03}c_{20}-8\,a_{03}a_{{11}
}-a_{11}c_{30}+2\,a_{12}a_{20}-a_{12}c_{11}-a_{{11}
}c_{12}-a_{20}c_{21}\right.\\
&-8\,a_{02}a_{12}+3\,a_{02}c_{{0,
3}}-3\,a_{02}c_{21}+2\,c_{03}c_{11}+3\,c_{11}c_{21}+
2\,c_{12}c_{20}+2\,a_{11}a_{21}-a_{20}c_{03}\\
&\left.-2\,b_{{
11}}+2\,d_{02}-2\,d_{{2,0}}+2\,c_{02}c_{30}-a_{21}c_{{02
}}-a_{21}c_{20} \right) {{\rm e}^{  \sin^2 \theta }}\sin \theta,\\
 \\
 A_{0,3}(\theta)=&- \left[ -a_{11}c_{21}-d_{11}-a_{12}c_{20}-3\,b_{02}
+b_{20}+2\,a_{11}a_{12}+3\,c_{02}c_{21}+2\,c_{03}c_{
{20}}-a_{20}c_{12}\right.\\
&\left.-a_{21}c_{11}-10\,a_{02}a_{03}+2
\,a_{02}a_{21}+2\,c_{11}c_{12}+2\,a_{03}a_{20}-a_{{0
2}}c_{30} \right] {{\rm e}^{  \sin^2 \theta }},\\
 \\
A_{0,2}(\theta)=&- \left[
2\,a_{03}a_{11}-a_{03}c_{20}-a_{02}c_{21}-a
_{{20}}c_{03}-a_{21}c_{02}-a_{11}c_{12}-a_{12}c_{{1
1}}-d_{02}-d_{{20}}\right.\\
&\left.+2\,c_{02}c_{12}+2\,c_{03}c_{11}+b
_{{11}}+2\,a_{02}a_{12} \right] {{\rm e}^{  \sin^2 \theta }}\sin \theta,\\
\\
A_{0,1}(\theta)=&- \left(
2\,c_{02}c_{03}-a_{02}c_{12}+b_{02}-d_{11}+
2\,a_{02}a_{03}-a_{03}c_{11}-a_{11}c_{03}-a_{12}c
_{{02}} \right) {{\rm e}^{  \sin^2 \theta  }},\\
\\
A_{0,0}(\theta)=&- \left[ -d_{02}-a_{02}c_{03}-a_{03}c_{02} \right]
{ {\rm e}^{  \sin^2 \theta }}\sin \theta.
\end{array}
}
$$

Additionally, we have

$$
\small{
\begin{array}{ll}
A_5(\theta)=&\left[\sin \theta\cos \theta\left( -1+2\, \cos^2 \theta
\right){a_{{00}}}^{2}
 + \left( -4\, \cos^4 \theta+2\, \cos^2 \theta
 +1 \right) c_{{00}}a_{{00}}\right.\\
&\left.-\sin  \theta\cos \theta
 \left( 2\, \cos^2 \theta+1 \right){c_{{00}}}^{2}
 \right]{{\rm e}^{6\, \sin^2 \theta}},\\
\end{array}
}
$$

$$
\small{
\begin{array}{ll}
A_3(\theta)=&2\,
 \cos^3 \theta\sin \theta\left( -1+2\, \cos^2 \theta
 \right)  a_{{00}}a_{{20}
} -2\,\cos \theta\sin \theta\left( \cos^2 \theta-1 \right)
 \left( -1+2\, \cos^2 \theta \right)
 a_{{00}}a_{{02}}\\
 &-2\, \cos^2 \theta \left( \cos^2 \theta-1 \right)
 \left( -1+2\, \cos^2 \theta \right)
 a_{{00}}a_{{11}}
 - \cos^2 \theta \left( 4\, \cos^4 \theta-1-2\, \cos^2 \theta \right)  a_{{0
0}}c_{{20}}\\
&+ \left( \cos^2 \theta-1 \right)   \left( 4\, \cos^4 \theta-1-2\,
\cos^2 \theta
 \right)  a_{{00}}c_{{02}}
-\cos \theta\sin \theta\left( 4\, \cos^4 \theta-1-2\, \cos^2 \theta \right)  c_{{11}}a_{{00}}\\
 &-\cos \theta\sin \theta\left( \cos^2 \theta-1 \right)    \left( -1+2\, \cos^2 \theta \right)  {a_{{01}}}^{2}
 - \left( \cos^2 \theta-1 \right)
   \frac{  2\,{I_1}-4\, \cos^4 \theta{I_3}+2\, \cos^2 \theta{I_3}-{I_2} }{ 2\,{
I_1 -{I_2}} } a_{{01}}c_{{01}}\\

 &-\cos \theta\sin \theta\left( 4\, \cos^4 \theta-1-2\, \cos^2 \theta \right) c_{{10}}a_{{01}}
-\cos^2 \theta \left( 4\, \cos^4 \theta
-1-2\, \cos^2 \theta\right)  c_{{00}}a_{{20}}\\
 &+ \left( \cos^2 \theta-1 \right)
  \left( 4\, \cos^4 \theta-1-2\, \cos^2 \theta \right)  c_{{00}}a_{{02}}
 -\cos \theta\sin \theta\left( 4\, \cos^4 \theta-1-2
\, \cos^2 \theta \right)  c_{{00}}a_{{11}}\\
&-2\, \cos^3 \theta\sin \theta\left( 2\, \cos^2 \theta+1 \right)
c_{{00}}c_{{20}}
 +2\,\cos \theta\sin \theta\left( \cos^2 \theta-1
\right)
   \left( 2\, \cos^2 \theta+1 \right)  c_{{00}}c_{{02}}\\
&+2\, \cos^2 \theta \left( \cos^2 \theta-1 \right)   \left( 2\,
\cos^2 \theta+1 \right)  c_{{11}}c_{{00}}
 -\cos \theta\sin \theta\, {I_3}\,
\frac{ 2\,{I_1}-{I_2}+ \cos^2 \theta{I_3}-2\, \cos^4 \theta{I_3} }{
2\,{
I_1}-{I_2} } {c_{{01}}}^{2}\\
&- \cos^2 \theta \,\frac{ 2\,{I_1}+{I_3}+2\, \cos^2
\theta{I_3}-{I_2}-4\, \cos^4 \theta{I_3} }{  2\,{ I_1}-{I_2}  }
c_{{10}}c_{{01}}
 -\left. \cos^3 \theta\sin \theta\left( 2\,
\cos^2 \theta+1 \right)  {c_{{10}}}^{2}\right]\exp(4\sin^2\theta),\\
\\
%\end{array}
%}
%$$
%
%
%$$
%\small{
%\begin{array}{ll}
A_{1}(\theta)=&\left[- \cos^4 \theta\frac{ 2\,{I_1}+{I_3}+2 \,
\cos^2 \theta{I_3}-{I_2}-4 \, \cos^4 \theta{I_3}}{2I_1-I_2} c_{{0
1}}c_{{30}}
+ \cos^4 \theta \left( 8\, \cos^4 \theta-3-4\, \cos^2 \theta \right)   c_{{10}}c_{{21}}\right.\\

& +2\, \cos^2 \theta
 \left( -1+2\, \cos^2 \theta \right)
 \left( \cos^2 \theta-1 \right) ^{2}
a_ {{02}}a_{{11}} + \cos^2 \theta \left( \cos^2 \theta-1 \right)
\left( 4\, \cos^4 \theta-1- 2\,
\cos^2 \theta \right)  a_{{20}}c_{{02}}\\

&+\sin \theta\cos \theta\left( -1+2\, \cos^2 \theta \right)  \left(
\cos^2 \theta-1 \right) ^{2}
   {a_{{02}}}^{2}

-\sin \theta\cos^3 \theta \left( \cos^2 \theta-1
 \right)   \left( -1+2\,
 \cos^2 \theta \right)   {a_{{11}}}^{2}\\

&-2\, \cos^4 \theta \left( \cos^2 \theta-1 \right)    \left( -1+2\,
\cos^2 \theta \right)  a_{{11}}a_{{20}}
 +{{\rm e}^{2\, \sin^2 \theta}} \cos^2 \theta
 \left( \cos^2 \theta-1 \right)    \left( 8\, \cos^4 \theta-4\, \cos^2 \theta-1
 \right)   a_{{01}}c_{{21}}\\
 &+2
\,
 \cos^2 \theta \left( -1+2\, \cos^2 \theta \right)  \left( \cos^2 \theta-1 \right) ^{2}  a_{{01}}a_{{12}
} +
 \cos^2 \theta \left( -1+2\, \cos^2 \theta \right)  b_{{10}}\\

&+ \cos^2 \theta
 \left( \cos^2 \theta-1 \right)    \left( 4\, \cos^4 \theta-1-2\, \cos^2 \theta
 \right)   c_{{10}}a_{{12}}
 +2
\,
 \cos^4 \theta \left( \cos^2 \theta-1 \right)
 \left( 2\, \cos^2 \theta+1 \right)
  c_{{11}}c_{{20}}\\
 &-\sin \theta\cos \theta\left( 2\, \cos^2 \theta+1 \right)  \left( \cos^2 \theta-1
 \right) ^{2}  {c_{{02}}}^{2}
-2\, \cos^2 \theta \left( 2\, \cos^2 \theta+1\right)
   \left( \cos^2 \theta-1 \right) ^{
2}  c_{{11}}c_{{02}}\\
&+\sin \theta\cos^3 \theta \left( \cos^2 \theta-1 \right)
 \left( 2\, \cos^2 \theta+1 \right)
  {c_{{11}}}^{2}
 + \cos^2 \theta \left( \cos^2 \theta-1
\right)
   \left( 4\, \cos^4 \theta-1-2\, \cos^2 \theta \right)  c_{{11}}a_{{11}}\\
\end{array}
%}
%$$
%
%$$
%\small{
%\begin{array}{ll}
-\sin \theta \cos^3 \theta \frac{4\, \cos^2 \theta{I_3}+4
\,{I_1}-8\, \cos^4 \theta{I_3}
+{I_3}-2\,{I_2}}{2I_1-I_2} c_{{01}}c_{{21}}\\

&+2\, \cos^2 \theta \left(  \cos^4 \theta - \cos^2 \theta-1+2\,
 \cos^6 \theta \right) c_{{03}}c_{{10}}

+\sin \theta \cos^5 \theta \left( -1+2\, \cos^2 \theta \right)  {a_{{20}}}^{2}\\

&- \cos^4 \theta
 \left( 4\, \cos^4 \theta-1-2\,
 \cos^2 \theta \right)   a_{{20}}c_{{20}}

+ \left( \cos^2 \theta- 1 \right) ^{2}  \frac{2\,{I_1}-4\, \cos^4
\theta{I_3 }+2\, \cos^2 \theta{I_3}-{I_2}
 }{2I_1-I_2} a_{{03}}c_{{01}}\\

 &- \left( \cos^2 \theta-1 \right)
   \left( 2\, \cos^2 \theta+1 \right)  d_{{01}}
+\sin \theta\cos \theta\left( 2\, \cos^2 \theta+1
 \right)   d_{{10}}+\sin \theta\cos \theta\left( -1+2\, \cos^2 \theta \right)   b_{{01}}\\

 &-\sin \theta\cos^3 \theta \left( 4\, \cos^4 \theta-1-2\, \cos^2 \theta
 \right)  a_{{11}}c_{{20}}
 -\sin \theta\cos^3 \theta
 \left( 4\, \cos^4 \theta-1-2\,
 \cos^2 \theta \right)   a_{{01}}c_{{30}}\\
&+ \left( \cos^2 \theta- 1 \right)    \left( 4\,
 \cos^6 \theta+2\, \cos^4 \theta- \left( \cos \theta\right)
^ {2}-1 \right)   a_{{01}}c_{{03
}}\\
&-
 \cos^2 \theta \left( \cos^2 \theta-1 \right)
\frac{  2\,{I_1}-4\, \cos^4 \theta {I_3}+2\, \cos^2 \theta{I_3}-{
I_2} }{2I_1-I_2} a_{{21}}c_{{01}}\\

&- \left( 4\, \cos^4 \theta-1-2\, \cos^2 \theta \right)  \left(
\cos^2 \theta-1 \right) ^{2}  a_{{02}}c_{{02}}

-\sin \theta\cos^3 \theta \left( 4\, \cos^4 \theta-1-2\, \cos^2
\theta
 \right)   c_{{11}}a_{{20}}\\
\end{array}
%}
%$$
%
%$$
%\small{
%\begin{array}{ll}
-2\,\sin \theta\cos^3 \theta
 \left( \cos^2 \theta-1 \right)   \left( -1+2\, \cos^2 \theta \right)  a_{{01}
}a_{{21}}\\
&+2\,\sin \theta\cos \theta\left( -1+2\,
 \cos^2 \theta \right)  \left( \cos^2 \theta-1 \right) ^{2}   a_{{01}}a_{{03
}} -2\, \sin \theta \cos^3 \theta \left( \cos^2 \theta-1 \right)
\left( -1+2\, \cos^2 \theta \right)
 a_{{02}
}a_{{20}}\\
&+\sin \theta\cos \theta\left( \cos^2 \theta-1 \right)   \left( 4\,
\cos^4 \theta-1-2
\, \cos^2 \theta \right)  c_{{11}}a_{{02}}\\
&+\sin \theta\cos \theta\left( \cos^2 \theta-1 \right)
   \left( 4\, \cos^4 \theta-1-2\, \cos^2 \theta \right)  a_{{11}}c_{{02}}\\
 &+\sin \theta\cos \theta\left( \cos^2 \theta-1 \right)   \left( 4\, \cos^4 \theta-1-2\, \cos^2 \theta \right)
c_{{10}}a_{{03} } +2\,\sin \theta\cos^3 \theta
 \left( \cos^2 \theta-1 \right)    \left( 2\, \cos^2 \theta+1 \right)   c_{{1
2}}c_{{10}}\\
&\left.+2\,\sin \theta\cos^3 \theta \left( \cos^2 \theta-1 \right)
   \left( 2\, \cos^2 \theta+1 \right)   c_{{02}}c_{{20}}\right]{{\rm e}^{2\, \sin^2 \theta}},
\end{array}
}
$$

$$
\small{
\begin{array}{ll}
\tilde{A}_1(\theta)=&\cos \theta\sin \theta
 \left(  \cos^4 \theta- \cos^2 \theta+2\,  \cos^8 \theta+3\, \cos^6 \theta-1 \right){c_{{03}}}^{2}\\
&- \left( 4\,  \cos^6 \theta+2\,
\cos^4 \theta- \cos^2 \theta-1 \right)  \left( \cos^2 \theta-1 \right) ^{2} a_{{03}}c_{{03}}\\

&-  \cos^5 \theta \sin \theta\left( 4\, \cos^4 \theta-1-2\, \cos^2
\theta
 \right) a_{{21}}c_{{30}}

 + \cos^6 \theta \left( 8\, \cos^4 \theta-3-4\,
 \cos^2 \theta \right) c_{{21}}c_{{30
}}\\
&+2\, \cos^4 \theta \left(  \cos^4 \theta- \cos^2 \theta-1+2\, \cos^6 \theta \right) c_{{03}}c_{{30}}\\

&+ \cos^3 \theta\sin \theta\left( 2\, \cos^2 \theta+1 \right)
d_{{30}} + \cos^3 \theta\sin \theta\left( -1+2\,
 \cos^2 \theta \right) b_{{21}}
 \\
 &+2\,{c
_{{21}}}^{2} \cos^5 \theta\sin \theta\left( 4\, \cos^4 \theta-1-2\,
\cos^2 \theta
 \right)
 +{a_{{12}}}^{2}  \sin^5 \theta
 \cos^3 \theta \left( -1+2\, \cos^2 \theta \right)\\
&-  \cos^7 \theta\sin \theta\left( 2\, \cos^2 \theta+1 \right)
{c_{{30}}}^{2} +
 \left( 2\, \cos^2 \theta+1 \right)
 \left( \cos^2 \theta-1 \right) ^{2} d_{{03}}\\
&+ \cos^3 \theta\sin \theta\left( 4\, \cos^4 \theta+8\, \cos^6
\theta-3-3\, \cos^2 \theta \right) c_{{03}}c_{{21}}
- \cos^2 \theta \left( \cos^2 \theta-1 \right)   \left( 2\, \cos^2 \theta+1 \right) d_{{21}}\\
\end{array}
%}
%$$
%
%$$
%\small{
%\begin{array}{l}

 &+2\, \cos^3 \theta\sin \theta\left( -1+
2\, \cos^2 \theta \right)
\left(\cos^2 \theta-1\right)  ^{2} a_{{03}}a_{{21}}\\

&+ \cos^3 \theta\sin \theta\left( \cos^2 \theta-1 \right)
 \left( 4\, \cos^4 \theta-1-2\,
 \cos^2 \theta \right) a_{{03}}c_{{30
}}\\
&+ \cos^3 \theta\sin \theta\left( \cos^2 \theta-1 \right)   \left( 8\, \cos^4 \theta-4\, \cos^2 \theta-1 \right) a_{{12}}c_{{21}}\\

&+ \cos^3 \theta\sin \theta\left( \cos^2 \theta-1 \right)
 \left( 4\, \cos^4 \theta-1-2\,
 \cos^2 \theta \right) a_{{21}}c_{{1
2}}\\
&-\cos \theta\sin \theta\left( 4\,
 \cos^4 \theta-1-2\, \cos^2 \theta \right)  \left( \cos^2 \theta-1 \right) ^{2} a_{{03}}c_{{12}}\\
&+\cos \theta\sin \theta\left( \cos^2 \theta-1 \right)   \left( 4\,
\cos^6 \theta+2\, \cos^4 \theta- \cos^2 \theta-1 \right)
a_{{12}}c_{{03}}.
\end{array}
}
$$
%%%%%%%%%%%%%%%%%%%%%%%%%%%%%%%%%%%%%%%%%%%%%%%%%%%%%%%%%%%%%%%%%%%%%%%%%%%%5

\section{Appendix C}
Here we present the explicit expressions of
$s_5(\theta),s_4(\theta),s_3(\theta),s_2(\theta),s_1(\theta),\tilde{s}_1(\theta)$
that appear in relation \eqref{ss}.
$$
s_5(\theta)=s_{5,1}(\theta)c_{00}^2+s_{5,2}(\theta)
a_{00}^2+s_{5,3}(\theta) a_{00}c_{00},
$$
with
$$
\small{
\begin{array}{ll}
s_{5,1}(\theta)=&-2\,{{\rm e}^{3\,  \sin^2  \theta }}
 \left(\displaystyle{ \int _{0}^{\theta}}\!{{\rm e}^{3\,  \sin^2  w  }}\sin  w  {dw}+2\,\int _{0}^{\theta}\!{
{\rm e}^{3\,  \sin^2  w  }}\sin  w    \cos^2  w {dw} \right)\\
 &\left( \sin  \theta  +2\,\sin \theta    \cos^2 \theta \right),\\
\\
s_{5,2}(\theta)=&-2\,{{\rm e}^{3\, \sin^2 \theta }}
 \left( -\displaystyle{\int _{0}^{\theta}}\!{{\rm e}^{3\,  \sin^2  w  }}\cos  w {dw}+2\,\int _{0}^{\theta}\!{
{\rm e}^{3\,  \sin^2  w  }}  \cos^3  w  {dw} \right) \\
 & \qquad \left( -\cos \theta +2\,  \cos^3 \theta \right),\\
\\
s_{5,3}(\theta)=&-2\,{{\rm e}^{3\,  \sin ^2 \theta }}
 \left(\displaystyle{ \int _{0}^{\theta}}\!{{\rm e}^{3\,  \sin^2  w  }}\sin  w  {dw}+2\,\int _{0}^{\theta}\!{
{\rm e}^{3\,  \sin^2  w  }}\sin  w    \cos^2  w {dw} \right)\\
& \qquad \left( -\cos  \theta  +2\,  \cos^3 \theta \right)\\

  & \qquad -2\,{{\rm e}^{3\, \sin^2
\theta }} \left( -\displaystyle{\int _{0}^{\theta}}\!{{\rm e}^{3\,
  \sin^2  w  }}\cos  w {dw
}+2\,\int _{0}^{\theta}\!{{\rm e}^{3\,  \sin^2  w  }}  \cos^3  w
{dw}
 \right)\\
 &\qquad   \left( \sin  \theta  +2\,\sin  \theta
   \cos^2 \theta \right).
\end{array}
}
$$
%%%%%%%%
Additionally, we have
$$
s_4(\theta)=s_{4,1}(\theta)a_{00}c_{10}-s_{4,2}(\theta)a_{00}a_{01}-s_{4,3}(\theta)c_{00}c_{01}-s_{4,4}(\theta)c_{00}c_{10}
-s_{4,5}(\theta)a_{00}c_{01}-s_{4,6}(\theta)a_{01}c_{00},
$$
and $s_{4,i}(\theta)$ for $i=1\cdots,6$ are the following:

$$
\small{
\begin{array}{ll}
s_{4,1}(\theta)=&- \left( - \sin \theta   {{\rm e}^{-  \sin^2
\theta }}\displaystyle{\int _{0}^{\theta}}\!{{\rm e}^{3\, \sin^2 w
}}\cos  w {dw}+4\,
  \cos^2  \theta  \displaystyle{\int _{0}^{\theta}}\!{
{\rm e}^{2\, \sin^2 w
}}\sin  w  \cos  w  {dw}\right.\\
& +2\, \sin \theta   { {\rm e}^{-  \sin^2  \theta  }}\int _{0}^{
\theta}\!{{\rm e}^{3\, \sin^2 w }}
  \cos^3 w  {dw}+4\,  \cos^2  \theta   \sin \theta   {{\rm e}^{-
  \sin^2  \theta  }}\displaystyle{\int _{0}^{\theta}}\!{
{\rm e}^{3\, \sin^2 w }}   \cos^3  w   {dw}\\
 &
+8\,  \cos^2  \theta  \displaystyle{\int _{0}^{\theta}}\!{{\rm
e}^{2\, \sin^2 w }} \cos^3 w  \sin  w  {dw}-2\,  \cos^2  \theta
  \sin \theta   {{\rm e}^{-  \sin^2  \theta  }}\displaystyle{\int _{0}^{\theta}}\!{{\rm e}^{3\,
 \sin^2 w }}\cos  w  {dw
}\\
&-\left.4\,\displaystyle{\int _{0}^{\theta}}\!{{\rm e}^{2\, \sin^2
w}} \cos^3 w  \sin  w  {dw}-2\,\displaystyle{\int
_{0}^{\theta}}\!{{\rm e}^{2\, \sin^2 w}}\sin  w  \cos  w  {d w}
\right) {{\rm e}^{3\, \sin^2  \theta
}}\cos \theta, \\
\\
s_{4,2}(\theta)=&- \left(  \sin \theta   {{\rm e}^{-  \sin^2  \theta
}}\displaystyle{\int _{0}^{\theta}}\!{{\rm e}^{3\, \sin^2 w }}\cos
w {dw}-2\, \sin \theta {{\rm e}^{-  \sin^2  \theta
}}\displaystyle{\int _{0}^{\theta}}\!{{\rm e}^{3\, \sin^2 w }}
\cos^3 w  {dw} \right.
\\
&-4\,  \cos^2  \theta  \displaystyle{\int _{0}^{\theta}}\!{ {\rm
e}^{2\, \sin^2 w }}\sin w \cos  w  {dw}-4\,\displaystyle{\int
_{0}^{\theta}}\!{{\rm e}^{2\,
 \sin^2 w }}  \cos^3 w  \sin  w  {dw}\\
 &-2\,  \cos^2  \theta   \sin \theta   {{\rm e}^{-
  \sin^2  \theta  }}\displaystyle{\int _{0}^{\theta}}\!{
{\rm e}^{3\, \sin^2 w }}\cos  w  {dw}\\
\end{array}
}
$$
\small{
$$
\begin{array}{ll}
s_{4,3}(\theta)=&-%{\frac {
{{\rm e}^{2\,  \sin^2  \theta  }}
 \left( 8\,{{\rm e}^{  \sin^2  \theta  }}{
I_1}\,  \cos^2  \theta   \sin \theta   \displaystyle{\int
_{0}^{\theta}}\!{{\rm
e}^{2\, \sin^2 w }}{dw}\right.\\
 &-2\,{I_2}\,\displaystyle{\int _{0}^{\theta}}\!{{\rm e}^
{3\, \sin^2 w }}\sin  w    \cos^2  w  {dw}+2\,{I_1}
\,\displaystyle{\int _{0}^{\theta}}\!{{\rm e}^{3\, \sin^2 w }}\sin  w  {dw}\\
 &-2\,{{\rm e}^{  \sin^2  \theta  }}{I_2}\, \sin \theta   \displaystyle{\int _{0}^{\theta}}\!{{\rm e}^{2\, \sin^2 w }}{dw}-2\, \cos ^4 \theta{
I_3}\,\displaystyle{\int _{0}^{\theta}}\!{{\rm e}^{3\, \sin^2 w }}\sin  w  {dw}\\
 &-4\,{{\rm e}^{
  \sin^2  \theta  }}{I_2}\,  \cos^2  \theta   \sin \theta   \int
_{0 }^{\theta}\!{{\rm e}^{2\, \sin^2 w }}{d w}+2\,{{\rm e}^{  \sin^2
\theta  }}{I_3}\, \sin \theta \displaystyle{\int
_{0}^{\theta}}\!{{\rm
e}^{2\, \sin^2 w }}  \cos^2  w {dw}\\

 &+4\,{I_1}\,\displaystyle{\int _{0}^{\theta}}\!{{\rm e}^{3\,
 \sin^2 w }}\sin  w
  \cos^2  w  {dw}
 -4\,{{\rm e}^{  \sin^2  \theta  }}{I_3}\, \sin \theta   \displaystyle{\int _{0}^{\theta}}\!{{\rm e}^{2\, \sin^2 w }}  \cos^4 w  {dw}\\
 &+2\,
  \cos^2  \theta  {I_3}\,\displaystyle{\int _{0}^{
\theta}}\!{{\rm e}^{3\, \sin^2 w }}\sin  w    \cos^2  w {dw}
 +4\,{
{\rm e}^{  \sin^2  \theta  }}{I_3}\,
  \cos^2  \theta   \sin \theta   \displaystyle{\int _{0}^{\theta}}\!{{\rm e}^{2\, \sin^2 w }}  \cos^2  w  {dw}\\
 &+4\,{
{\rm e}^{  \sin^2  \theta  }}{I_1}\,  \sin \theta
\displaystyle{\int _{0}^{\theta}}\!{{\rm e}^{2\, \sin^2 w }}{dw}+
\cos^2  \theta {I_3}\,\displaystyle{\int _{0}^{\theta}}\!{{\rm
e}^{3\,
 \sin^2 w }}\sin  w  {dw
}\\
&-{I_2}\,\displaystyle{\int _{0}^{\theta}}\!{{\rm e}^{3\, \sin^2 w
}}\sin  w {dw}-4\, \cos ^4 \theta{I_3}\,\displaystyle{\int
_{0}^{\theta}}\!{
{\rm e}^{3\, \sin^2 w }}\sin  w    \cos^2  w  {dw}\\
\end{array}
$$
} \small{
$$
\begin{array}{ll}
s_{4,4}(\theta)=&- \left( 8\,{{\rm e}^{  \sin^2  \theta  }}
  \cos^2  \theta  \displaystyle{\int _{0}^{\theta}}\!{
{\rm e}^{2\, \sin^2 w }}  \cos^3  w  \sin  w  {dw}\right.\\
 &+4\,{{\rm e}^
{  \sin^2  \theta  }}  \cos^2  \theta \displaystyle{\int
_{0}^{\theta}}\!{{\rm e}^{2\, \sin^2 w }}\sin  w  \cos  w  {dw}
 +
 2\left(\,\displaystyle{\int _{0}^{\theta}}\!{{\rm e}^{3\, \sin^2 w }}\sin  w   \cos^2  w  {dw}\right)\cos  \theta \\

 &+\left(\displaystyle{\int _{0}^{\theta}}\!
{{\rm e}^{3\, \sin^2 w }}\sin  w  {dw}\right)\cos  \theta
 +2\,{{\rm e}^{  \sin^2  \theta  }}\displaystyle{\int _{0}^{\theta}}\!{{\rm e}^{2\,
 \sin^2 w }}\sin  w
\cos  w  {dw}\\

&+4\left(\,\displaystyle{\int _{0}^{\theta}}\!{{\rm e}^{3\, \sin^2 w
}}\sin w \cos^2 w {dw} \right) \cos^3 \theta +4\,{{\rm e}^{ \sin^2
\theta }}\displaystyle{\int _{0}^{\theta}}\!{{\rm e}^{2\, \sin^2 w
}} \cos^3 w
\sin  w  {dw}\\

&\left.+2\left(\,\displaystyle{\int _{0}^{\theta}}\!{{\rm e}^{3\,
\sin^2 w }}\sin  w  {dw}\right)
  \cos^3 w   \right) {{\rm e}^{2\,
  \sin^2  \theta  }} \sin \theta,
\end{array}
$$
}

\small{
$$
\begin{array}{ll}
s_{4,5}(\theta)=&-%{\frac {
{{\rm e}^{2\,  \sin^2  \theta  }}
 \left( -2\,{I_1}\,\displaystyle{\displaystyle{\int _{0}^{\theta}}}\!{{\rm e}^{3\, \sin^2 w }}\cos  w  {dw}+2\,
 \cos ^4 \theta{I_3}\,\int _{0}^{
\theta}\!{{\rm e}^{3\, \sin^2 w }}\cos  w  {dw} \right.\\

 &-2\,{I_2}\,\displaystyle{\displaystyle{\displaystyle{\int _{0}^{\theta}}}}\!{{\rm e}^{3\,
 \sin^2 w }}  \cos^3 w  {dw}+4\,{{\rm e}^{  \sin^2  \theta  }}{I_3}\,\cos  \theta  \int _{0}^
{\theta}\!{{\rm e}^{2\, \sin^2 w }}
  \cos^4 w  {dw}\\
 &-2\,{{\rm e}^{  \sin^2  \theta  }}{I_3}\,\cos \theta
  \displaystyle{\displaystyle{\displaystyle{\int _{0}^{\theta}}}}\!{{\rm e}^{2\, \sin^2 w }} \cos^2  w   {dw}-4\,
 \cos ^4 \theta{I_3}\,\int _{0}^{
\theta}\!{{\rm e}^{3\, \sin^2 w
 }}
  \cos^3 w  {dw}\\
 &+2\,{{\rm e}^{  \sin^2  \theta  }}{I_2}\,\cos  \theta
  \displaystyle{\displaystyle{\displaystyle{\int _{0}^{\theta}}}}\!{{\rm e}^{2\, \sin^2 w }}{dw}+{I_2}\,\displaystyle{\displaystyle{\displaystyle{\int _{0}^{\theta}}}}\!{{\rm e}^{3\, \sin^2 w }}\cos  w  {dw}\\
&-4\,{ {\rm e}^{  \sin^2  \theta  }}{I_2}\,
  \cos^3 w  \displaystyle{\displaystyle{\displaystyle{\int _{0}^{\theta}}}}\!{
{\rm e}^{2\, \sin^2 w }}{dw}-8\,{ {\rm e}^{  \sin^2  \theta
}}{I_3}\,
  \cos^3 w  \displaystyle{\displaystyle{\displaystyle{\int _{0}^{\theta}}}}\!{
{\rm e}^{2\, \sin^2 w }}   \cos^4  w {dw}\\
 &+4\,{{\rm e}^{  \sin^2  \theta  }}{I_3}\,  \cos^3 w  \displaystyle{\displaystyle{\displaystyle{\int _{0}^{\theta}}}}\!{{\rm e}^{2\, \sin^2 w }} \cos^2  w
 {dw}+4\,{I_1}\,\displaystyle{\displaystyle{\displaystyle{\int _{0}^{\theta}}}}\!{{\rm e}^{3\,
 \sin^2 w }}  \cos^3 w  {dw}\\
 &+8\,{{\rm e}^{  \sin^2  \theta  }}{I_1}\,  \cos^3 w  \displaystyle{\displaystyle{\displaystyle{\int _{0}^{\theta}}}}\!{{\rm e}^{2\, \sin^2 w }}{dw}
 +2\,  \cos^2  \theta  {I_3}\,\displaystyle{\displaystyle{\displaystyle{\int _{0}^{\theta}}}}\!{{\rm e}^{3\, \sin^2 w }}  \cos^3 w  {dw}\\
\end{array}
$$
}
$$
\small{
\begin{array}{ll}
s_{4,6}(\theta)=&- \left(
4\,\displaystyle{\displaystyle{\displaystyle{\int
_{0}^{\theta}}}}\!{{\rm e}^{2\, \sin^2 w }}  \cos^3 w \sin  w
{dw}+4\, \cos^3 w  {{\rm e}^{-  \sin^2  \theta
}}\displaystyle{\displaystyle{\displaystyle{\int
_{0}^{\theta}}}}\!{{\rm e}^{3\, \sin^2 w }}\sin  w  \cos^2  w {dw}\right.\\
& +8\,  \cos^2  \theta
\displaystyle{\displaystyle{\displaystyle{\int
_{0}^{\theta}}}}\!{{\rm e}^{2\, \sin^2 w }}  \cos^3 w \sin  w
{dw}-4\,  \cos^2  \theta  \int _{0 }^{\theta}\!{{\rm e}^{2\, \sin^2
w }}
\sin  w  \cos  w  {dw}\\
&+2\,  \cos^3 w  {{\rm e}^{-  \sin^2  \theta
}}\displaystyle{\displaystyle{\displaystyle{\int
_{0}^{\theta}}}}\!{{\rm e}^{3\, \sin^2 w }}\sin  w  {dw}-\cos
 \theta  {{\rm e}^{-  \sin^2  \theta  }}\displaystyle{\displaystyle{\displaystyle{\int _{0}^{\theta}}}}\!{{\rm e}^{3\, \sin^2 w }}\sin  w  {dw}\\
& -\left. 2\,\cos  \theta
  {{\rm e}^{-  \sin^2  \theta  }}
\displaystyle{\displaystyle{\displaystyle{\int
_{0}^{\theta}}}}\!{{\rm e}^{3\, \sin^2 w }}\sin  w \cos^2  w
{dw}-2\,\displaystyle{\displaystyle{\displaystyle{\int
_{0}^{\theta}} }} \!{{\rm e}^{2\, \sin^2 w }}\sin w
 \cos  w  {dw}
 \right) {{\rm e}^{3\,  \sin^2  \theta}}
 \sin \theta.
\end{array}
}
$$
%%%%%%%%%%%%%%%%%%%%%%%%%%%%%%%%%%%%%%%%

Now we have
$$
\begin{array}{l}
s_3(\theta)=s_{3,1}(\theta) a_{00}c_{02}
+s_{3,2}(\theta)a_{02}c_{00} +s_{3,3}(\theta)a_{01}c_{01}
+s_{3,4}(\theta)a_{00}a_{11}
+s_{3,5}(\theta)a_{00}c_{20}\\
+s_{3,6}(\theta)c_{00}c_{11} +s_{3,7}(\theta)c_{01}c_{10}
+s_{3,8}(\theta)a_{00}a_{02}
+s_{3,9}(\theta)a_{00}a_{20}\\
+s_{3,10}(\theta)a_{00}c_{11} +s_{3,11}(\theta)a_{01}c_{10}
+s_{3,12}(\theta)a_{11}c_{00} +s_{3,13}(\theta)c_{00}c_{02}
+s_{3,14}(\theta)c_{00}c_{20}\\
+s_{3,15}(\theta)a_{01}^2 +s_{3,16}(\theta)c_{10}^2
+s_{3,17}(\theta)c_{01}^2 +s_{3,18}(\theta)a_{20}c_{00},
\end{array}
$$
and $s_{3,i}(\theta)$ for $i=1\cdots,18$ satisfying the following
expressions:

\small{
$$
\begin{array}{ll}
s_{3,1}(\theta)=&-2\,{{\rm e}^{3\,  \sin^2  \theta }}\cos
\theta\left[\left( 4\,\displaystyle{\int _{0}^{\theta}}\!{{\rm e}^{
\sin^2 w }} \sin^3 w  \cos^2 w {dw}\right)
 \cos^2  \theta \right.\\
&\left.-2\,\displaystyle{\int _{0}^{\theta}}\!{{\rm e}^{
 \sin^2 w }} \sin^3 w  \cos^2 w {dw}+
2\,\left(\displaystyle{\int _{0}^{\theta}}\!{{\rm e}^{ \sin^2 w }}
\sin^3 w {dw}\right) \cos^2  \theta -\displaystyle{\int
_{0}^{\theta}}\!{{\rm e}^{
 \sin^2 w }} \sin^3 w {dw} \right],\\
 \\
 s_{3,2}(\theta)=&2\,{{\rm e}^{3\, \sin^2 \theta}}\sin
 \theta   \left[ 2\,\left(\displaystyle{\int _{0}^{\theta}}\!{{\rm e}^{ \sin^2 w }}\cos w {dw}\right)
 \cos^2  \theta \right.\\
 &+2\,\displaystyle{\int _{0}^{\theta}}\!{
{\rm e}^{ \sin^2 w }}  \cos^5 w {dw}+4\,\left(\displaystyle{\int
_{0}^{\theta}}\!{{\rm e}^{
 \sin^2 w }} \cos^5 w {dw}\right) \cos^2 \theta +\displaystyle{\int _{0}^{\theta}}\!{{\rm e}^{ \sin^2 w
^{2}}}\cos w {dw}\\
&-\left.6\,\left(\displaystyle{\int _{0}^{\theta}}\!{{\rm e}^{
 \sin^2 w }} \cos^3 w  {dw}\right) \cos^2 \theta -3\,\displaystyle{\int _{0}^{\theta}}\!{{\rm e}^{ \sin^2 w }} \cos^3
w  {dw}
 \right),\\
\\

s_{3,3}(\theta)=&-
%{\frac {
{{\rm e}^{2\,  \sin^2 \theta }}
 \left[ -2\,{I_1}\,\cos \theta\sin  \theta
  \displaystyle{\int _{0}^{\theta}}\!{{\rm e}^{2\, \sin^2 w }}{dw}
 +4\,{I_1}\,\displaystyle{\int _{0}^{\theta}}\!{{\rm e}^{2\,
 \sin^2 w }} \cos^3 w

\sin w {dw} \right.\\
\end{array}
%$$
%}
%$$
%\small{
%\begin{array}{ll}
s_{3,4}(\theta)=&2\,{{\rm e}^{3\, \sin^2  \theta }}\cos
\theta\left[\left( 2\,\displaystyle{\int _{0}^{\theta}}\!{{\rm e}^{
\sin^2 w }} \cos^2 w \sin w {dw}\right)
\cos^2  \theta  \right.\\
& -\displaystyle{\int _{0}^{\theta}}\!{{\rm e}^{ \sin^2 w }} \cos^2
w \sin w {dw}+2\,\displaystyle{\int _{0}^{\theta}}\!{{\rm e}^
{ \sin^2 w }}  \cos^4 w \sin w {dw}\\
 &\left.-4\,\left(\displaystyle{\int _{0}^{\theta}}\!
{{\rm e}^{ \sin^2 w }}  \cos^4 w \sin w {dw} \right)\cos^2  \theta  \right],\\
\\
\end{array}
$$
}
$$
\small{
\begin{array}{ll}
s_{3,5}(\theta)=&-2\,{{\rm e}^{3\, \sin^2  \theta }}\cos
\theta\left[ 2\,\left(\displaystyle{\int _{0}^{\theta}}\!{{\rm e}^{
\sin^2 w }} \cos^2 w \sin w {dw}\right)
\cos^2  \theta  \right.\\
 &+4\,\left(\displaystyle{\int _{0}^{\theta}}\!{{\rm e}^{ \sin^2 w }}  \cos^4 w \sin w {dw}\right) \cos^2  \theta \\
 &-\left.\displaystyle{\int _{0}^{\theta}}\!{{\rm e}^{ \sin^2 w }} \cos^2 w \sin w {dw}-2\,\displaystyle{\int _{0}^{\theta}}\!{{\rm e}^
{ \sin^2 w }}  \cos^4 w \sin w {dw} \right],\\
 \\
 s_{3,6}(\theta)=&-2\,{{\rm e}^{3\, \sin^2  \theta }}\sin \theta   \left[ \displaystyle{\int _{0}^{\theta}}\!{{\rm e}^{ \sin^2 w }}\cos w {dw}
 -2\,\displaystyle{\int _{0
}^{\theta}}\!{{\rm e}^{ \sin^2 w }}
  \cos^5 w {dw}\right.\\
 &+\displaystyle{\int _{0}^{\theta}}\!{
{\rm e}^{ \sin^2 w }} \cos^3 w  {dw}+2\,\left(\displaystyle{\int
_{0}^{\theta}}\!{{\rm e}^{
 \sin^2 w }}\cos w {dw
}\right) \cos^2  \theta \\
& +\left.2\,\left(\displaystyle{\int _{0}^{\theta}}\!{ {\rm e}^{
\sin^2 w }} \cos^3 w  {dw}\right) \cos^2  \theta
-4\,\left(\displaystyle{\int _{0}^{\theta}}\!{{\rm e}^{ \sin^2 w }}
\cos^5 w {dw }\right)
\cos^2  \theta  \right],\\
\\
%\end{array}
%}
%$$
%
%$$
%\small{
%\begin{array}{ll}
s_{3,7}(\theta)=&-
%{\frac {
{{\rm e}^{2\,  \sin^2 \theta }}
 \left[ 2\,{I_1}\,\cos \theta\sin \theta \displaystyle{\int _{0}^{\theta}}\!{{\rm e}^{2\, \sin^2 w }}{dw}
 +4\,{I_1}\,\sin \theta \left(\displaystyle{\int _{0}^{
\theta}}\!{{\rm e}^{2\, \sin^2 w }}{dw}\right)
  \cos^3 \theta \right.\\
 &+2\,{I_1}\,\displaystyle{\int _{0}
^{\theta}}\!{{\rm e}^{2\, \sin^2 w }} \sin w \cos w
{dw}+4\,{I_1}\,\displaystyle{\int _{0 }^{\theta}}\!{{\rm e}^{2\,
\sin^2 w }}
 \cos^3 w  \sin w {dw}\\
 &+
{I_3}\,\cos \theta\sin \theta \displaystyle{\int _{0}
^{\theta}}\!{{\rm e}^{2\, \sin^2 w }}
 \cos^2 w {dw}-{I_2}\,\cos \theta\sin \theta \displaystyle{\int _{0}^{\theta}}\!{
{\rm e}^{2\, \sin^2 w }}{dw}\\
&-2\,{I_3}\,\cos \theta\sin \theta \displaystyle{\int _{0}^{
\theta}}\!{{\rm e}^{2\, \sin^2 w }}
  \cos^4 w {dw}
 -{I_2}\,\int _{0}^{
\theta}\!{{\rm e}^{2\, \sin^2 w }}\sin w \cos w {dw}\\
 &+{I_3}\,\left(\displaystyle{\int _{0}^{
\theta}}\!{{\rm e}^{2\, \sin^2 w }}\sin w \cos w {dw}\right) \cos^2
\theta
 -2\,{I_2}\,\sin \theta \left(\displaystyle{\int _{0
}^{\theta}}\!{{\rm e}^{2\, \sin^2 w }}{d
w} \right) \cos^3 \theta\\
&-2\,{I_2}\,\displaystyle{\int _{0 }^{\theta}}\!{{\rm e}^{2\, \sin^2
w }}
 \cos^3 w  \sin w {dw}-
4\,{I_3}\,\sin \theta \left(\displaystyle{\int _{0}^{\theta}}\!{{\rm
e}^{2\,
 \sin^2 w }}  \cos^4 w {dw} \right) \cos^3 \theta\\
&+2\,{I_3}\,\sin \theta \left(\displaystyle{\int
_{0}^{\theta}}\!{{\rm e}^{ 2\, \sin^2 w }} \cos^2 w {dw}\right)
\cos^3 \theta +2\,{I_3}\,\left(\displaystyle{\int
_{0}^{\theta}}\!{{\rm e}^{2\, \sin^2 w }} \cos^3 w   \sin w
{dw}\right)
\cos^2  \theta\\

&\left.-4\,{I_3}\,\left(\displaystyle{\int _{0}^{\theta}}\!{{\rm
e}^{2\, \sin^2 w  }} \cos^3 w \sin w {dw}\right)  \cos^4
\theta-2\,{I_3}\,\left(\displaystyle{\int _{0}^{\theta}}\!{{\rm
e}^{2\, \sin^2 w }}\sin w \cos w {d w}\right) \cos^4 \theta
\right] \dfrac{1}{2\,{I_1}-{I_2}},\\
\end{array}
}
$$
\small{
$$
\begin{array}{ll}
s_{3,9}(\theta)&=2\,{{\rm e}^{3\, \sin^2  \theta }}\cos \theta\left[
2\,\displaystyle{\int _{0}^{\theta}}\!{{\rm e}^{ \sin^2 w }}\cos^5 w
 {dw}-4\,\left(\displaystyle{\int _{0}^{\theta}}\!{{\rm e}^{ \sin^2 w }} \cos^5 w {dw
} \right)\cos^2  \theta \right.\\
&-\left.\displaystyle{\int _{0}^{\theta}}\!{ {\rm e}^{ \sin^2 w }}
\cos^3 w  {dw}+2\,\left(\displaystyle{\int _{0}^{\theta}}\!{{\rm
e}^{
 \sin^2 w }} \cos^3 w  {dw}\right)  \cos^2 \theta\right],\\
\\
s_{3,10}(\theta)&=-2\,{{\rm e}^{3\, \sin^2  \theta }}\cos
\theta\left[ 2\,\left(\displaystyle{\int _{0}^{\theta}}\!{{\rm e}^{
\sin^2 w }} \cos^3 w  {dw}\right)
 \cos^2  \theta \right.\\
 &+2\,
\left(\displaystyle{\int _{0}^{\theta}}\!{{\rm e}^{ \sin w ^{2
}}}\cos w {dw}\right) \cos^2  \theta

 -\displaystyle{\int _{0}^{\theta}}\!{{\rm e}^{ \sin^2 w }} \cos^3 w  {dw
}\\
&\left. -4\,\left(\displaystyle{\int _{0}^{\theta}}\!{{\rm e}^{
\sin^2 w }}
 \cos^5 w {dw}\right) \cos^2  \theta -\displaystyle{\int
_{0}^{\theta}}\!{{\rm e}^{
 \sin^2 w }}\cos w {dw
}+2\,\displaystyle{\int _{0}^{\theta}}\!{{\rm e}^{ \sin^2 w }}
\cos^5 w {dw}
 \right],\\
%\end{array}
%$$
%}
%
%\small{
%$$
%\begin{array}{ll}\\
s_{3,11}(\theta)=& 2\,{{\rm e}^{2\,  \sin^2 \theta }}\sin
 \theta \cos \theta\left[ -4\,\left(\displaystyle{\int _{0}^{
\theta}}\!{{\rm e}^{2\, \sin^2 w }}
 \cos^3 w  \sin w {dw}\right)
 \cos^2  \theta \right.\\
 &+\left.\displaystyle{\int _{0}^{\theta}}\!{
{\rm e}^{2\, \sin^2 w }}\sin w \cos w {dw} \right],\\
 \\

s_{3,12}(\theta)=&2\,{{\rm e}^{3\, \sin^2  \theta }}\sin \theta
\left[ 2\,\left(\displaystyle{\int _{0}^{\theta}}\!{{\rm e}^{ \sin^2
w }}
 \cos^2 w \sin w {dw}\right) \cos^2  \theta \right.\\
 &+\displaystyle{\int _{0}^{\theta}}\!{{\rm e}^{ \sin^2 w }} \cos^2 w \sin w {dw}-2\,\displaystyle{\int _{0}^{\theta}}\!{{\rm e}^
{ \sin^2 w }}  \cos^4 w  \sin w {dw}\\

 &\left.-4\,\left(\displaystyle{\int _{0}^{\theta}}\!
{{\rm e}^{ \sin^2 w }}  \cos^4 w \sin w {dw}\right) \cos^2  \theta  \right],\\
 \\
%\end{array}
%$$
%}
%
%\small{
%$$
%\begin{array}{ll}
 s_{3,13}(\theta)=&-2\,{{\rm e}^{3\, \sin^2  \theta }}\sin
 \theta   \left[ 2\,\displaystyle{\int _{0}^{\theta}}\!{{\rm e}^{ \sin^2 w }} \sin^3 w  \cos^2 w {dw}\right.\\
 &+\displaystyle{\int _{0}
^{\theta}}\!{{\rm e}^{ \sin^2 w }}
 \sin^3 w {dw}+4\,\left(\displaystyle{\int _{0}^{\theta}}\!{
{\rm e}^{ \sin^2 w }} \sin^3 w  \cos^2 w {dw} \right)\cos^2  \theta \\
&+\left.2\,\left(\displaystyle{\int _{0}^{ \theta}}\!{{\rm e}^{
\sin^2 w }} \sin^3 w {dw}\right)
\cos^2  \theta  \right],\\
 \\
s_{3,14}(\theta)=& -2\,{{\rm e}^{3\, \sin^2  \theta }}\sin
 \theta  \left[ 2\,\displaystyle{\int _{0}^{\theta}}\!{{\rm e}^{ \sin^2 w }}  \cos^4 w \sin w {dw}\right.\\
 &+4\,\left(\displaystyle{\int _{0}^{\theta}}\!{{\rm e}^
{ \sin^2 w }}  \cos^4 w \sin w {dw} \right)\cos^2  \theta
+\displaystyle{\int _{0}^{\theta}}\!{{\rm e}^{ \sin^2 w }}
\cos^2 w \sin w {dw}\\
& +\left.
 2\,\left(\displaystyle{\int _{0}^{\theta}}\!{{\rm e}^
{ \sin^2 w }} \cos^2 w \sin w {dw} \right)\cos^2  \theta \right],
\end{array}
$$
} \small{
$$
\begin{array}{ll}
s_{3,15}(\theta)=&{{\rm e}^{2\,  \sin^2 \theta }}\sin
 \theta \cos \theta\left[\left( -4\,\displaystyle{\int _{0}^{
\theta}}\!{{\rm e}^{2\, \sin^2 w }}
 \cos^3 w  \sin w {dw}\right)
 \cos^2  \theta \right.\\
& -\displaystyle{\int _{0}^{\theta}}\!{ {\rm e}^{2\, \sin^2 w }}\sin
w \cos w {dw}+2\,\left(\displaystyle{\int _{0}^{\theta}}\!{{\rm
e}^{2\,
 \sin^2 w }}\sin w
\cos w {dw} \right)\cos^2  \theta\\
&+\left.2\,\displaystyle{\int _{0}^{\theta}}\!{{\rm e}^{2\, \sin^2 w}} \cos^3 w  \sin w {dw} \right),\\
\\
s_{3,16}(\theta)=&-{{\rm e}^{2\,  \sin^2 \theta }}\sin
 \theta \cos \theta\left[ 4\,\left(\displaystyle{\int _{0}^{
\theta}}\!{{\rm e}^{2\, \sin^2 w }}
 \cos^3 w  \sin w {dw}\right)
 \cos^2  \theta \right.\\
 &+
 \displaystyle{\int _{0}^{\theta}}\!{
{\rm e}^{2\, \sin^2 w }}\sin w \cos w
{dw}+2\,\left(\displaystyle{\int _{0}^{\theta}}\!{{\rm e}^{2\,
 \sin^2 w }}\sin w
\cos w {dw}\right) \cos^2 \theta \\
&+\left.2\,\displaystyle{\int _{0}^{\theta}}\!{{\rm e}^{2\, \sin^2 w}} \cos^3 w  \sin w {dw} \right],\\
%\end{array}
%$$
%}
%
%\small{
%$$
%\begin{array}{ll}
\\
s_{3,17}(\theta)=&-
%{\frac {
{{\rm e}^{2\,  \sin^2 \theta }}
 \left[ -4\,{I_1}\,{I_2}\,\displaystyle{\int _{0}^{\theta}}\!{{\rm e}^{2\,
 \sin^2 w }}{dw}
 -2\,{{I_3}}^{2}
\left(\displaystyle{\int _{0}^{\theta}}\!{{\rm e}^{2\, \sin^2 w }}  \cos^4 w {dw}\right) \cos^2  \theta \right.\\
& -{I_2}\,{I_3}\,\displaystyle{\int _{0}^{ \theta}}\!{{\rm e}^{2\,
\sin^2 w }}
 \cos^2 w {dw}
 +{{I_2}}^{2}\displaystyle{\int _{0
}^{\theta}}\!{{\rm e}^{2\, \sin^2 w }}{d
w}\\

 &+2\,{I_1}\,{I_3}\,\displaystyle{\int _{0}^{\theta}}\!{{\rm e}^{2\, \sin^2 w }} \cos^2 w {dw}
 +2\,{I_2}\,{I_3}\,\displaystyle{\int _{0}^{\theta}}\!{{\rm e}^
{2\, \sin^2 w }}  \cos^4 w {dw}\\
 &+2\,{I_1}\,{I_3}\,\left(\displaystyle{\int _{0}^{\theta}}\!
{{\rm e}^{2\, \sin^2 w }}{dw}\right) \cos^2  \theta
+{{I_3}}^{2}\left(\displaystyle{\int _{0}^{\theta}} \!{{\rm e}^{2\,
\sin^2 w
}} \cos^2 w {dw} \right)\cos^2  \theta \\
 &-{I_2}\,{I_3}\,\left(\displaystyle{\int _{0}^{\theta}}\!{
{\rm e}^{2\, \sin^2 w }}{dw} \right)\cos^2  \theta
+4\,{{I_1}}^{2}\int _{0}^{ \theta}\!{{\rm e}^{2\, \sin^2 w }}{dw}-
4\,{I_1}\,{I_3}\,\displaystyle{\int _{0}^{\theta}}\!{{\rm e}^{2\,
\sin^2 w }}  \cos^4 w {dw}\\
& -2\,{{I_3}}^{2}\left(\displaystyle{\int _{0}^{\theta}}\!{{\rm
e}^{2\,
 \sin^2 w }} \cos^2 w

 {dw} \right) \cos^4 \theta +2\,{I_2}\,{I_3}\,\left(\displaystyle{\int _{0}^{\theta}}\!{{\rm e}^{2\,
\sin^2 w }}{dw}\right) \cos^4 \theta \\
 &\left.
 -4\,{I_1}\,{I_3}\,\left(\displaystyle{\int _{0}^{\theta}}\!{
{\rm e}^{2\, \sin^2 w }}{dw}\right) \cos^4 \theta
+4\,{{I_3}}^{2}\left(\displaystyle{\int _{0}^{ \theta}}\!{{\rm
e}^{2\, \sin^2 w }}
 \cos^4 w {dw}\right) \cos^4 \theta  \right]
\dfrac{1}{ \left( -{I_2}+2\,{I_1}\right) ^{2}},\\
\end{array}
$$
}
%%%%%%%%%%%%%%%%%%%%%%%%%%%%%%%%%%%%%%%%%%%%%%%%%%%%%%%%%%%%%%%%%%%%%
Now we have
$$
\small{
\begin{array}{ll}
s_2(\theta)=&s_{2,1}(\theta)\,a_{{00}}c_{{30}}+s_{2,2}(\theta)\,a_{{20}}c_{{10}}+s_{2,3}(\theta)\,a_{{00}}c_{{12}}+s_{2,4}(\theta)\,a_{{02}}c_{{10}}
+s_{2,5}(\theta)
\,a_{{12}}c_{{0,0}}+s_{2,6}(\theta)\,a_{{00}}a_{{21}}\\
&+s_{2,7}(\theta)\,a_{{0
1}}a_{{20}}+s_{2,8}(\theta)\,a_{{01}}c_{{11}}+s_{2,9}(\theta)\,a_{{11}}c_
{{01}}+s_{2,10}(\theta)\,c_{{00}}c_{{21}}+s_{2,11}(\theta)\,c_{{01}}c_{{20
}}+s_{2,12}(\theta)\,c_{{10}}c_{{11}}\\
&+s_{2,13}(\theta)\,a_{{00}}a_{{03}}+
s_{2,14}(\theta)\,a_{{01}}a_{{02}}+s_{2,15}(\theta)\,c_{{00}}c_{{03}}+s_{2,16}(\theta)\,c_{{01}}c_{{02}}+s_{2,17}(\theta)\,a_{{00}}a_{{12}}+s_{2,18}(\theta)\,a_{{03}}c_{{00}}
\\
&+s_{2,19}(\theta)\,a_{{11}}c_{{10}}+s_{2,20}(\theta)\,a_{{01}}c_{{02}}+s_{2,21}(\theta)\,a_{{01}}c_{{20}}+s_{2,22}(\theta)\,a_{{01}}a_{{11}}+s_{2,23}(\theta)\,a_{{00}}c_{{03}}
\\
&+s_{2,24}(\theta)\,a_{{00}}c_{{21}}+s_{2,25}(\theta)\,a_{{02}}c_{{01}}+s_{2,26}(\theta)\,a_{{20}}c_{{01}}+s_{2,27}(\theta)\,a_{{21}}c_{{00}}\\
&+s_{2,28}(\theta)\,c_{{00}}c_{{12}}
+s_{2,29}(\theta)\,c_{{00}}c_{{30}}+s_{2,30}(\theta)\,c_{{01}}c_{{11}}+s_{2,31}(\theta)\,c_{{02}}c_{{10}}+s_{2,32}(\theta)\,c_{{10}}c_{{20}},
\end{array}
}
$$
and $s_{2,i}(\theta)$ for $i=1,\cdots,32$ satisfying the following
expressions:
$$
\small{
\begin{array}{ll}
s_{2,1}(\theta)=&\dfrac{1}{6}\, \left( 12\, \cos^2 \theta \sin
\theta  {{\rm e}^{-3\, \sin^2 \theta}}\displaystyle{\int
_{0}^{\theta}}\!{{\rm e}^{3\, \sin^2 w }} \cos^3 w {dw }\right.
-12\, \cos^4 \theta \sin  \theta  \left(\displaystyle{\int _{0}^{\theta}}\!{{\rm e}^{3\, \sin^2 w }}\cos w{dw}\right){{\rm e}^{-3\, \sin^2 \theta}}\\
&+2\, \cos^6 \theta
 -14\,  \cos^2 \theta
-6\, \cos^2 \theta \sin  \theta  \left(\displaystyle{\int _{0}^{\theta}}\!{{\rm e}^{3\, \sin^2 w }}\cos w{dw}\right){{\rm e}^{-3\, \sin^2 \theta}}\\
 &+24\, \cos^4 \theta \sin  \theta  {{\rm e}^{-3\, \sin^2 \theta}}\displaystyle{\int _{0}^{\theta}}\!{{\rm e}^{3
\, \sin^2 w }} \cos^3 w {dw}-3\, \cos^4 \theta +7
 +\left.8\, \cos^8 \theta
 \right) {{\rm e}^{3\, \sin^2 \theta}}
\cos \theta,\\
\\
s_{2,2}(\theta)=&\left( -2\,\displaystyle{\int _{0}^{\theta}}\!{{\rm
e}^{ \sin^2 w }} \cos^5 w{dw}+\displaystyle{\int _{0
}^{\theta}}\!{{\rm e}^{ \sin^2 w }}
 \cos^3 w {dw}\right.\\
 &+\left.2\,\left(\displaystyle{\int _{0}^{\theta}}\!{
{\rm e}^{ \sin^2 w }} \cos^3 w {dw}\right) \cos^2 \theta
-4\,\left(\displaystyle{\int _{0}^{\theta}}\!{{\rm e}^{ \sin^2 w }}
\cos^5 w{dw }\right) \cos^2 \theta  \right) {{\rm e}^{2\,
 \sin^2 \theta}}\sin  \theta\cos \theta,
 \end{array}%
}
$$

%%%%%%%%%%%%%%%%%%%%%%%%%%%%%%%%%%%%%%%%%%%%%%%%%%%%%%%%%%%%%%%%%%%%%%%%%%%%%%%%
We also have
$$
\small{
%
}
$$
and $s_{1,i}(\theta)$ for $i=1\cdots,21$ are given in the following
expressions:
$$
\small{
%
}
$$
%%%%%%%%%%%%%%%%%%%%%%%%%%%%%%%%%%%%%%%%%%%%%%%%%%%555

Now we present the expression of $s_0(\theta)$.
$$
\small{
%
}
$$
and $s_{0,i}(\theta)$ for $i=1\cdots,42$ are the following:

$$
\small{
%
}
$$
%%%%%%%%%%%%%%%%%%%%%%%%%%%%%%%%%%%%%%%%%%%%%%%%%%%%%%%%%%%%%

Now we have
$$
\small{
%
}
$$
and $\tilde{s}_{1,i}(\theta)$ for $i=1\cdots,28$ satisfying the
following expressions.

$$
\small{
%
 }
$$

\end{document}